\documentclass[11pt]{article}

\usepackage{amsmath}
\usepackage{amssymb, amsthm}
\usepackage{algorithm, algpseudocode}
\usepackage{graphicx}
\usepackage{fullpage}

\def\<{\langle}
\def\>{\rangle}

\def\eps{{\varepsilon}}

\def\id{{\rm I}}
\def\sT{{\sf T}}
\def\proj{{\cal P}}

\def\projnd{{\mathcal{P}_{\sf nd}}}

\def\sign{{\operatorname{\rm{sgn}}}}

\def\P{{\mathbb P}}
\def\prob{{\mathbb P}}
\def\tprob{\widetilde{\mathbb P}}

\def\naturals{{\mathbb N}}
\def\E{{\mathbb E}} 
\def\tE{\widetilde{\mathbb E}} 

\def\der{\partial}

\def\betamin{\beta_{\rm min}}

\def\reals{\mathbb{R}}

\def\normal{{\sf N}}

\def\cA{{\cal A}}

\def\grad{{\nabla}}

\def\de{{\rm d}}

\def\obX{\overline{\mathbf{X}}}

\def\cG{\mathcal{G}}
\def\d{{\mathrm{d}}}

\def\ind{\mathbb{I}}

\newcommand\norm[1]{\lVert{#1}\rVert}
\newcommand\abs[1]{\lvert{#1}\rvert}
\def\bsl{\backslash}

\newcommand\myeqref[1]{{Eq.\,\eqref{#1}}}

\def\sC{{\sf Q}}
\def\sQ{{\sf Q}}

\def\hsC{{\widehat{\sf Q}}}

\def\bG{{\bf G}}

\def\cU{{\mathcal{U}}}
\def\ed{\stackrel{{\rm d}}{=}}

\def\convD{{\,\stackrel{\mathrm{d}}{\Rightarrow} \,}}

\def\sD{{\sf D}}
\def\sE{{\sf E}}

\def\sF{{\sf F}}
\def\sG{{\sf G}}
\def\sS{{\sf S}}

\def\ba{{\mathbf a}}
\def\be{{\mathbf e}}

\def\bx{{\mathbf x}}
\def\by{{\mathbf y}}
\def\bv{{\mathbf v}}
\def\bz{{\mathbf z}}
\def\bzdot{{\dot{\mathbf{z}}}}
\def\bu{{\mathbf u}}
\def\bw{{\mathbf w}}
\def\bX{{\mathbf X}}
\def\bA{{\mathbf A}}
\def\bZ{{\mathbf Z}}
\def\bG{{\mathbf G}}
\def\bS{{\mathbf S}}
\def\bN{{\mathbf N}}
\def\bP{{\mathbf P}}
\def\btN{{\mathbf{\widetilde{N}}}}

\def\bC{{\mathbf C}}
\def\bD{{\mathbf D}}
\def\sphere{{\mathbb{S}}}

\renewcommand{\b}[1]{\mathbf{#1}}
\newcommand{\bt}[1]{\mathbf{\tilde{#1}}}

\def\bsigma{{\boldsymbol{\sigma}}}
\def\bmu{{\boldsymbol{\mu}}}
\def\beps{{\boldsymbol{\eps}}}

\def\bta{{\mathbf{\tilde{a}}}}
\def\btz{{\mathbf{\tilde{z}}}}

\def\hbv{\mathbf{\widehat{v}}}
\def\hv{\widehat{v}}

\def\bfone{{\boldsymbol 1}}
\def\bU{{\mathbf U}}
\def\bB{{\mathbf B}}
\def\bV{{\mathbf V}}
\def\bfe{{\mathbf e}}
\def\btC{\bar{\mathbf C}}
\def\cU{{\mathcal U}}
\def\bP{{\mathbf P}}
\def\bSigma{{\mathbf \Sigma}}
\def\bDelta{{\mathbf \Delta}}
\def\hbSigma{{\mathbf{\widehat{\Sigma}}}}
\def\supp{{\rm supp}}
\def\MAD{{\rm MAD}}
\def\sPCA{\mbox{\tiny{\sc PC}}}

\newtheorem{claim}{Claim}[section]
\newtheorem{lemma}[claim]{Lemma}

\newtheorem{theorem}{Theorem}
\newtheorem{proposition}[claim]{Proposition}
\newtheorem{corollary}[claim]{Corollary}
\newtheorem{definition}[claim]{Definition}
\theoremstyle{definition}
\newtheorem{remark}[claim]{Remark}

\begin{document}

\title{Sparse PCA via Covariance Thresholding}

\author{Yash~Deshpande\thanks{Department of Electrical Engineering, Stanford University}
\and Andrea~Montanari\thanks{Departments of Electrical Engineering and Statistics, Stanford University}
     }
\maketitle

\begin{abstract}
In sparse principal component analysis we are given noisy observations
of a low-rank matrix of dimension $n\times p$ and seek to 
reconstruct it under additional sparsity assumptions.
In particular, we assume here each of  the principal components
$\bv_1,\dots,\bv_r$  has at most
$s_0$ non-zero entries. We are particularly interested in
the high dimensional regime wherein $p$ is comparable to, or even much larger
than $n$. 

In an influential paper, \cite{johnstone2004sparse} introduced a simple
algorithm that estimates the support of the principal vectors
$\bv_1,\dots,\bv_r$ by the largest entries
in the diagonal of the empirical covariance. 
This method can be shown to identify the correct support with high probability if
$s_0\le K_1\sqrt{n/\log p}$, and to fail with high probability if $s_0\ge
K_2 \sqrt{n/\log p}$ for two constants $0<K_1,K_2<\infty$. 
Despite a considerable amount of work over the last ten years, no
practical  algorithm exists with provably better support recovery guarantees.

Here we analyze a covariance thresholding algorithm that was recently
proposed by \cite{KrauthgamerSPCA}. On the basis of numerical simulations (for the rank-one case), 
these authors conjectured that covariance thresholding correctly recover the support with high probability for 
$s_0\le K\sqrt{n}$ (assuming $n$ of the same order as $p$).
We prove this conjecture, and in fact establish a more general guarantee including higher-rank
as well as $n$ much smaller than $p$.  
Recent lower bounds \cite{berthet2013computational, ma2015sum} suggest that no polynomial time 
algorithm can do significantly better.

The key technical component of our analysis develops new bounds on the norm of kernel
random matrices, in regimes that were not considered before. Using these, we also derive sharp bounds
for estimating the population covariance (in operator norm), and the principal component (in $\ell_2$-norm).
\end{abstract}

\newpage

\section{Introduction}

In the spiked covariance model proposed by \cite{johnstone2004sparse}, we are given data 
$\bx_1,\bx_2,\dots,\bx_n$ with $\bx_i\in \reals^p$ of the form\footnote{Throughout the paper, we follow the convention of denoting
scalars by lowercase, vectors by lowercase boldface, and
matrices by uppercase boldface letters.}:
\begin{align}
  \bx_i &= \sum_{q=1}^{r}\sqrt{\beta_{q}}\, u_{q,i} \,\bv _q+ \bz_i\,, \label{eq:model}
\end{align}
Here $\bv_1,\dots, \bv_r \in\reals^p$ is a set of orthonormal vectors, that we want to
estimate, while $u_{q,i} \sim\normal(0, 1)$ and $\bz_i \sim\normal(0, \id_p)$
are independent and identically distributed. The quantity $\beta_q\in
\reals_{>0}$ is a measure of signal-to-noise ratio. 
In the rest of this introduction, in order to simplify
the exposition, we will refer to the rank one case and drop the subscript $q\in\{1,2,\dots,r\}$. 
Further, we will assume $n$ to be of the same order as $p$.
Our results and proofs hold for a broad range of scalings of $r$, $p$, $n$, and will be stated in general form.

The standard method of principal component analysis  involves computing the
sample covariance matrix $\bG = n^{-1}\sum_{i=1}^n\bx_i\bx_i^{\sT}$ and estimates
$\bv=\bv_1$ by its principal eigenvector $\bv_{\sPCA}(\bG)$. It is a well-known
fact that, in the high dimensional regime, this yields an inconsistent estimate
(see \cite{johnstone2009consistency}). Namely $\|\bv_{\sPCA}-\bv\|\not\to 0$  unless $p/n \to 0$.
Even worse, \cite{baik2005phase} and  \cite{paul2007asymptotics} demonstrate the
following phase transition 
phenomenon. Assuming that $p/n\to \alpha \in (0, \infty)$,  if $\beta< \sqrt{\alpha}$ the estimate is \emph{asymptotically
orthogonal} to the signal, i.e. $\<\bv_{\sPCA},\bv\>\to 0$. On the other hand, for $\beta>\sqrt{\alpha}$,
$|\<\bv_{\sPCA},\bv\>|$ remains bounded away from zero as $n,p\to\infty$.
This phase transition phenomenon has attracted considerable attention
recently within random matrix theory
\cite{feral2007largest,capitaine2009largest,benaych2011eigenvalues,knowles2013isotropic}. 

These inconsistency results motivated several efforts to exploit
additional structural information on the signal $\bv$. 
In two influential papers,
\cite{johnstone2004sparse,johnstone2009consistency} considered the
case of a signal $\bv$ that is sparse in a suitable basis, e.g. in the
wavelet domain. Without loss of generality, we will assume here that $\bv$
is sparse in the canonical basis $\be_1$, \dots $\be_p$. 
In a nutshell, \cite{johnstone2009consistency} propose the following:
\begin{enumerate}
\item Order the diagonal entries of the Gram matrix
  $\bG_{i(1),i(1)}\ge \bG_{i(2),i(2)}\ge\dots\ge\bG_{i(p),i(p)}$, and let
  $J\equiv \{i(1),i(2),\dots,i(k)\}$ be the set of indices
  corresponding to the $s_0$ largest entries.
\item Set to zero all the entries $\bG_{i,j}$ of $\bG$ unless $i,j\in
  J$, and estimate $\bv$ with the principal eigenvector of the
  resulting matrix.
\end{enumerate}  
Johnstone and Lu formalized the sparsity assumption by requiring that $\bv$ belongs to a weak $\ell_q$-ball
with $q\in (0,1)$. 
\cite{amini2009high} studied 
the more restricted case when every entry of $\bv$ has equal
magnitude of $1/\sqrt{s_0}$. 
Within this restricted model, they proved
diagonal thresholding successfully recovers  the support of $\bv$
provided the sample size $n$ satisfies\footnote{Throughout the introduction, we write
  $f(n)\gtrsim g(n)$ as a shorthand of \emph{`$f(n)\ge K\, g(n)$ for a
  some constant $K = K(r,\beta)$'}.} $n \gtrsim s_0^2\log p$  
\cite{amini2009high}.
This result is a striking improvement over vanilla
PCA. While the latter requires a number of samples scaling with the
number of parameters $n \gtrsim  p$, sparse PCA via diagonal thresholding achieves the
same objective with a number of samples that scales with the
number of \emph{non-zero} parameters, $n\gtrsim s_0^2\log p$.

At the same time,  this result is not as strong as might have
been expected. By searching exhaustively over all possible supports
of size $s_0$ (a method that has complexity of order $p^{s_0}$)  the correct
support can be identified with high probability as soon as $n\gtrsim
s_0\log p$.  No method can succeed for much smaller
$n$, because of information theoretic obstructions. We refer the reader
to \cite{amini2009high} for more details. 

Over the last ten years, a significant effort has been devoted to developing practical
algorithms that outperform diagonal thresholding, see e.g.
\cite{moghaddam2005spectral,zou2006sparse,d2007direct,d2008optimal,witten2009penalized}.
In particular,  \cite{d2007direct} developed a
promising M-estimator based on a semidefinite programming (SDP)
relaxation. \cite{amini2009high} also carried out an
analysis of this method and proved that, if\footnote{Throughout the paper,
we denote by $K$ constants that can depend on problem parameters $r$
and $\beta$. We denote by upper case $C$ (lower case $c$) generic absolute constants that
are bigger (resp. smaller) than 1, but which might change from line 
to line.} \emph{(i)} $n\ge K(\beta)\, s_0\log(p-s_0)
p$,  and \emph{(ii)} the SDP solution has rank one, then the SDP
relaxation  provides a
consistent estimator of the support of  $\bv$. 

At first sight, this appears as a satisfactory solution of the
original problem.
No procedure can estimate the support of $\bv$ from less than $s_0\log
p$ samples, and the SDP relaxation succeeds in doing it from --at most-- a constant
factor more samples.
This picture was upset by a recent, remarkable result by \cite{KrauthgamerSPCA} who showed that
the rank-one condition assumed by Amini and Wainwright
does not hold for $ \sqrt{n}\lesssim s_0\lesssim (n/\log p)$. This
result is consistent with recent work of 
\cite{berthet2013computational}  demonstrating that sparse PCA cannot be
performed in polynomial time in the regime $s_0\gtrsim \sqrt{n}$, under
a certain computational complexity conjecture for the so-called
planted clique problem.

In summary, the sparse PCA problem demonstrates a fascinating
interplay between computational and statistical barriers.
\begin{description}
\item[From a statistical perspective,] and disregarding computational
  considerations, the support of $\bv$ can be estimated consistently 
 if and only if $s_0\lesssim n/\log p$. This can be done, for instance,
 by exhaustive search over all the $\binom{p}{s_0}$ possible supports of
 $\bv$. We refer to \cite{vu2012minimax,cai2013sparse} for a minimax analysis.
\item[From a computational perspective,] the problem appears to be
  much more difficult. There is rigorous evidence
  \cite{berthet2013computational,ma2015computational,ma2015sum, wang2014statistical} 
  that no polynomial
  algorithm can reconstruct the support unless $s_0\lesssim \sqrt{n}$. 
On the positive side, a very simple algorithm (Johnstone and Lu's
diagonal thresholding) succeeds for $s_0\lesssim \sqrt{n/\log p}$.
\end{description}
Of course, several elements are still missing in this emerging
picture. In the present paper we address one of them, providing
an answer to the following question:
\begin{quote}
\emph{Is there a polynomial time algorithm that is guaranteed to solve
  the sparse PCA problem with high probability for $\sqrt{n/\log
    p}\lesssim s_0\lesssim \sqrt{n}$?}
\end{quote}

We answer this question positively by analyzing a covariance
thresholding algorithm that proceeds, briefly, as follows.
(A precise, general definition, with some technical changes is given
in the next section.)
\begin{enumerate}
\item Form the empirical covariance matrix $\bG$ and set to zero all its entries that
  are in modulus smaller than $\tau/\sqrt{n}$, for $\tau$ a suitably
  chosen constant. 
\item Compute the principal eigenvector $\hbv_1$ of this thresholded
  matrix.
\item Estimate the support of $\bv$ by  thresholding $\bG\hbv_{1}$.
\end{enumerate}
Such a covariance thresholding approach was proposed in
\cite{KrauthgamerSPCA}, and is in turn related to earlier work by
\cite{bickel2008regularized, cai2010optimal}. The formulation
discussed in the next section presents some technical differences that
have been introduced to simplify the analysis. Notice that, to
simplify proofs, we assume $s_0$ to be known: this issue is discussed in
the next two sections.

The rest of the paper is organized as follows. In the next section we
provide a detailed description of the algorithm and state our main
results. The proof strategy for our results is explained in 
Section \ref{sec:proofStrategy}. Our theoretical results assume full knowledge of problem
parameters for ease of proof. In light of this, in Section \ref{sec:practical} we discuss a practical implementation
of the same idea that does not require prior knowledge of problem parameters, and is
data-driven. We also illustrate the method through simulations.
The complete proofs are in 
Sections \ref{sec:prelim}, \ref{sec:proofmain} and \ref{sec:proofcorr} respectively.

A preliminary version of this paper appeared in \cite{deshpande2014sparse}. 
This paper extends significantly the results in \cite{deshpande2014sparse}. 
In particular, by following an analogous strategy, we improve greatly the
bounds obtained by \cite{deshpande2014sparse}. This significantly improves the 
regimes of $(s_0, p, n)$ on which we can obtain non-trivial results. 
The proofs follow a similar strategy but are,
correspondingly, more careful.

%
%
\section{Algorithm and main results}

\begin{algorithm}
  \caption{Covariance Thresholding}
  \label{alg:ct}
  \begin{algorithmic}[1]
    \State {\bf Input:} Data $(\bx_i)_{1\le i\le 2n}$, parameter
   $s_0\in \naturals$, $\tau,\rho\in \reals_{\ge 0}$;
   \State Compute the empirical covariance matrices $\bG\equiv \sum_{i=1}^n\bx_i\bx_i^{\sT}/n$ , 
$\bG' \equiv \sum_{i=n+1}^{2n}   \bx_i\bx_i^\sT/n$;
    \State Compute $\hbSigma = \bG - \id_p$ (resp. $\hbSigma' = \bG'-\id_p$);
    \State Compute the matrix $\eta(\hbSigma)$  by soft-thresholding
    the entries of $\hbSigma$:
    \begin{align*}
      \eta(\hbSigma)_{ij}
&= \begin{cases}
        \hbSigma_{ij}-\frac{\tau}{\sqrt{n}} & \mbox{if
         $\hbSigma_{ij}\ge \tau/\sqrt{n}$,}\\
0& \mbox{if $-\tau/\sqrt{n}<\hbSigma_{ij}< \tau/\sqrt{n}$,}\\
  \hbSigma_{ij}+\frac{\tau}{\sqrt{n}} & \mbox{if
          $\hbSigma_{ij}\le -\tau/\sqrt{n}$,}
\end{cases}
   \end{align*}
   \State  Let $(\hbv_{q})_{q\le r}$ be the first $r$ eigenvectors of $\eta(\hbSigma)$;
   \State {\bf Output:} $\hsC = \{i\in [p]: \;\exists\, q \text{ s.t. } |(\hbSigma'\hbv_q)_i|\ge \rho \}$.
  \end{algorithmic} 
\end{algorithm}

We provide a detailed description of the covariance 
thresholding algorithm for the general model (\ref{eq:model}) in Table \ref{alg:ct}. 
For notational convenience, we shall assume that $2n$ sample vectors are given (instead of $n$): 
$\{\bx_i\}_{1\le i\le 2n}$. 

We start by splitting the data into  two halves: $(\bx_i)_{1\le i\le n}$ and
$(\bx_{i})_{n< i\le 2n}$ and compute the respective sample covariance matrices $\bG$ and
$\bG'$ respectively. Define $\bSigma$ to be the population covariance minus identity, i.e. 
\begin{align}
\bSigma \equiv \sum_{q=1}^r\beta_q\bv_q\bv_q^{\sT}\, \, .
\end{align}
Throughout, we let $\sC_q$ and $s_q$ denote the support of $\bv_q$ and its size respectively,
for $q\in\{1,2,\dots,r\}$.
We further let $\sC = \cup_{q=1}^r\sC_q$ and $s_0 = |\sC|$. 
The matrix $\bG$ is used, in steps $1$ to $4$ to obtain a good estimate $\eta(\hbSigma)$ 
for the low rank part of the population covariance $\bSigma$. 
The algorithm first computes $\hbSigma$, a centered version of the 
empirical covariance  as follows:
\begin{align}
    \hbSigma &\equiv \bG - \id_p,
\end{align}
where $\bG = n^{-1}\sum_{i\le n} \bx_i\bx_i^\sT$ is the sample covariance
matrix. 

It then obtains the estimate $\eta(\hbSigma)\in \reals^{p\times p}$
by \emph{soft thresholding} each entry of $\hbSigma$ at a threshold $\tau/\sqrt{n}$. 
Explicitly:
\begin{align}
    \big(\eta(\hbSigma)\big)_{ij}   &\equiv \eta\bigg( \hbSigma_{ij} ; \frac{\tau}{\sqrt{n}}\bigg)\, .
\end{align}
Here $\eta : \reals\times\reals_+ \to \reals$ is the soft thresholding function
\begin{align}
    \eta(z ; \lambda) &= \begin{cases}
        z - \lambda &\text{ if } z \ge \lambda \\
        -z + \lambda &\text{ if } z \le - \lambda \\
        0 &\text{ otherwise.}
    \end{cases}
\end{align}

In step $5$ of the algorithm, this estimate is used to construct good estimates $\hbv_q$ of the 
eigenvectors $\bv_q$. Finally, in step $6$, these estimates are combined with the (independent) second half of the data $\bG'$
to construct estimators $\hsC_q$ for the support of the individual eigenvectors 
$\bv_q$. In the first two subsections we will focus on the estimation of $\bSigma$
and the individual principal components. Our results on support recovery are provided
in the final subsection. 

\subsection{Estimating the population covariance}

Our first result bounds the estimation error of the 
soft thresholding procedure in operator norm.
\begin{theorem}    \label{thm:est}
There exist numerical constants $C_1, C_2, C>0$ such that the following happens.
Assume $n>C\log p$, $n>s_0^2$ and let $\tau_* = C_1(\beta\vee 1)\sqrt{\log (p/s_0^2)}$. We
keep the thresholding level $\tau$ according to
\begin{align}
    \tau = \begin{cases}
        \tau_* &\text{ when }\tau_* \le \sqrt{\log p}/2,\, s_0^2 \le p/e \\
        C_2 \tau_* &\text{ when } \tau_* \ge \sqrt{\log p}/2,\, s_0 \le p/e \\
        0 &\text{ otherwise.}
    \end{cases}
\end{align}
. Then 
  with probability $1-o(1)$:
    \begin{align}
\big\|\eta(\hbSigma) - \bSigma \big\|_{op}&\le C\,
\sqrt{\frac{s_0^2( \beta^2\vee 1)}{n}\,\Big(\log\frac{p}{s_0^2}\vee 1\Big)}\, .\label{eq:EstTheorem}
    \end{align}
\end{theorem}
At this point, 
it is useful to compare Theorem \ref{thm:est} with available results in the literature. 
Classical denoising theory \cite{DJ94a,johnstone2013function} provides upper bounds on the 
estimation error of soft-thresholding. However, estimation error is measured by (element-wise)  $\ell_p$ norm, while here we are interested
in operator norm.

\cite{bickel2008covariance, bickel2008regularized, karoui2008operator, cai2010optimal, cai2011adaptive} considered
  the  operator norm error of thresholding estimators for structured  covariance matrices. 
  Specializing to our case of exact sparsity, the result of \cite{bickel2008covariance} implies that, with high probability:
\begin{align}
      \big\|\eta_H(\hbSigma)-\bSigma \big\|_{op} 
\le      C_0\sqrt{\frac{s_0^2\log p}{n}}\,  .  \label{eq:bickelresult} 
\end{align}
Here $\eta_H(\cdot, \cdot)$ is the hard-thresholding function: $\eta_H(z) = z\ind(\abs{z} \ge \tau/\sqrt{n})$, and the threshold is chosen to
be $\tau= C_1\sqrt{\log p}$. Also, $\eta_H(\b{M})$  is the matrix obtained by thresholding the entries of $\b{M}$.
In fact, \cite{cai2012optimal} showed that the rate in (\ref{eq:bickelresult}) is minimax optimal over the class of sparse
population covariance matrices, with at most $s_0$ non-zero entries per row, under the 
assumption $s_0^2/n\le C(\log p)^{-3}$. 

Theorem \ref{thm:est} ensures consistency under 
a weaker sparsity condition, viz. $s_0^2/n\to 0$ 
is sufficient. Also, the resulting rate depends on $\log(p/s_0^2)$ instead of $\log p$. 
In other words, in order to achieve 
$\|\eta(\hbSigma)-\bSigma\|_{op}<\eps$ for a fixed $\eps$, it is sufficient $s_0\lesssim \eps\sqrt{n}$ as opposed to $s_0\lesssim \sqrt{n/\log p}$.

Crucially, in this regime for $s_0 = \Theta(\eps\sqrt{n})$, Theorem \ref{thm:est}  suggests a threshold of order 
$\tau =\Theta(\sqrt{\log (1/\eps)})$ as opposed to $\tau= C_1\sqrt{\log p}$ which is used in
\cite{bickel2008covariance,cai2012optimal}.
As we will see in Section \ref{sec:proofStrategy}, this regime mathematically more challenging than the one of \cite{bickel2008covariance,cai2012optimal}. By setting
$\tau= C_1\sqrt{\log p}$ for a large enough constant $C_1$, all the entries of $\hbSigma$ outside the support
of $\bSigma$ are set to $0$. In contrast, a large part of our proof is devoted to control the operator norm of the noise part of 
$\hbSigma$.

\subsection{Estimating the principal components}

We next turn to the question of estimating the principal components 
$\bv_1,\dots\bv_r$. Of course, these are not identifiable if
there are degeneracies in the population 
eigenvalues $\beta_1,\beta_2,\dots,\beta_r$. We thus introduce the following identifiability condition.
\begin{enumerate}
\item[{\sf A1}] The spike strengths $\beta_1>\beta_2>\dots\beta_r$ are all \emph{distinct}. 
 We denote by $\beta \equiv \max(\beta_1,\dots,\beta_r)$ and 
$\betamin \equiv  \min_{q\ne q'} (\beta_1-\beta_2,\beta_2-\beta_3,\dots,\beta_r)$. 
    Namely, $\beta$ is the largest signal strength and $\betamin$ is the minimum gap.
\end{enumerate}

We measure estimation error through the following loss, defined for 
$\bx,\by\in S^{p-1}\equiv \{\bv\in\reals^p:\; \|\bv\|=1\}$:
\begin{align}
L(\bx,\by) &\equiv \frac{1}{2}\, \min_{s\in \{+1,-1\}}\big\|\bx-s\, \by\big\|^2\\
  & = 1-|\<\bx,\by\>|\, .
\end{align}
Notice the minimization over the sign $s\in\{+1,-1\}$. This is required because the principal components
$\bv_1,\dots,\bv_r$ are only identifiable up to a sign. Analogous results can obtained for alternate
loss functions such as the projection distance:
\begin{align}
    L_p(\bx, \by) &\equiv \frac{1}{\sqrt{2}}\norm{\bx\bx^\sT - \by\by^\sT }_F = \sqrt{1-\<\bx, \by\>^2}.
\end{align}

The theorem below is an immediate consequence of Theorem \ref{thm:est}.
In particular, it uses the guarantee of Theorem \ref{thm:est} to 
show that the corresponding principal components of $\eta(\hbSigma)$
provide good estimates of the principal components $\bv_q, 1\le q \le r$.  
\begin{theorem}\label{thm:corr}
There exists a numerical constant $C$ such that the following holds.
    Suppose that Assumption {\sf A1} holds in addition to the conditions
     $n>C\log p$, $s_0^2<n$, and $s_0^2<p/e$.    
     Set $\tau$ as according to Theorem \ref{thm:est}, and 
    let $\hbv_1, \dots, \hbv_r$ denote the $r$ principal eigenvectors of $\eta(\hbSigma;\tau/\sqrt{n})$.
  Then, with probability $1-o(1)$
\begin{align}
    \max_{q\in [r]} L(\hbv_q,\bv_q)\le \frac{C}{\betamin^2}\; \frac{s_0^2( \beta^2\vee 1)}{n}\,\log\frac{p}{s_0^2}\, .
\end{align}
\end{theorem}
\begin{proof}
Let $\bDelta \equiv \eta(\hbSigma;\tau/\sqrt{n})-\bSigma$.
By Davis-Kahn sin-theta theorem \cite{davis1970rotation}, we have, for $\betamin>\|\bDelta\|_{op}$,
\begin{align}
L(\hbv_q,\bv_q)\le \frac{1}{2}\left(\frac{\|\bDelta\|_{op}}{\betamin-\|\bDelta\|_{op}}\right)^2\, .
\end{align}
For $\betamin^2>2C(s_0^2( \beta^2\vee 1)/n)\,\log(p/s_0^2)$, the claim follows by using Theorem \ref{thm:est}.
If $\betamin^2\le 2C(s_0^2( \beta^2\vee 1)/n)\,\log(p/s_0^2)$, the claim is obviously true since $L(\hbv_q,\bv_q)\le 1$ always.
\end{proof}

\subsection{Support recovery}

Finally, we consider the question of support recovery of the principal
components $\bv_q$. The second phase of our algorithm aims at estimating union of the
supports $\sC = \sC_1\cup \dots\cup\sC_r$ from the estimated principal
components $\hbv_q$. Note that, although $\hbv_q$ is 
not even expected to be sparse, it is easy to see that
the largest entries of $\hbv_q$ should have significant overlap
with  $\supp(\bv_q)$. Step 6 of the  algorithm exploit this 
property to construct a consistent estimator
$\hsC_q$ of the support of the spike $\bv_q$.

We will require the following assumption to ensure support recovery. 

\begin{enumerate}
\item[{\sf A2}] 
There exist constants $\theta,\gamma>0$ such that the following holds.
The non-zero entries of the spikes satisfy $|v_{q,i}|\ge \theta/\sqrt{s_0}$  for all $i\in \sC_q$. Further, for any $q, q'$
 $\abs{v_{q, i}/v_{q', i}} \le \gamma$ for every
 $i\in\sC_q\cap\sC_{q'}$. Without loss of generality, we will assume
 $\gamma\ge 1$.
\end{enumerate}

\begin{theorem}\label{thm:main}
  Assume the spiked covariance model of \myeqref{eq:model} satisfying
  assumptions {\sf A1} and {\sf A2}, and further $n>C\log p$, $s_0^2<n$, and $s_0^2<p/e$
for $C$ a large enough numerical constant. Consider the Covariance
Thresholding algorithm of Table \ref{alg:ct}, with $\tau$ as in Theorem \ref{thm:est} 
$\rho= \betamin\theta/(2\sqrt{s_0})$.

Then there exists $K_0 = K_0(\theta,\gamma,\beta,\betamin)$ such that, if
\begin{align}
n\ge K_0 s_0^2 r \log \frac{p}{s_0^2} \label{eq:Nsupport}
\end{align}
then the algorithm recovers the union of supports of $\bv_q$ with probability $1-o(1)$
(i.e. we have $\hsC = \sC$). 
\end{theorem}
The proof in Section \ref{sec:proofmain} also provides an explicit expression for
the constant $K_0$. 

\begin{remark}
    In Assumption {\sf A2}, the requirement on the minimum size of $|v_{q, i}|$ is 
    standard in support recovery literature \cite{wainwright2009sharp, meinshausen2006high}. Additionally, however, we require that when the supports of $\bv_q, \bv_{q'}$ overlap, they
    are of the same order, quantified by the parameter $\gamma$. Relaxing this condition
    is a potential direction for future work.
\end{remark}

\begin{remark}
  Recovering the signed supports $\sC_{q,+} = \{i\in[p] : v_{q, i} > 0\}$ and
  $\sC_{q,-} = \{i\in[p]: v_{q,i} <0\}$, up to a sign flip, is possible using the same technique
  as recovering the supports $\supp(\bv_q)$ above, and poses no additional difficulty.
\end{remark}

%
%
%

\section{Algorithm intuition and proof strategy}
\label{sec:proofStrategy}

For the purposes of exposition, throughout this section, 
we will assume that $r=1$ and drop the corresponding subscript
$q$. 

Denoting by $\bX\in\reals^{n\times p}$ the matrix with rows $\bx_1$,
\dots $\bx_n$, by $\bZ\in\reals^{n\times p}$ the matrix with rows $\bz_1$,
\dots $\bz_n$, and letting $\bu = (u_1,u_2,\dots,u_n)$, the model
(\ref{eq:model}) can be rewritten as
\begin{align}\label{eq:model2}
  \bX &= \sqrt{\beta}\, \bu \,\bv^{\sT} + \bZ\, .
\end{align}
Recall that $\hbSigma = n^{-1}\bX^\sT\bX-\id_p = \bG - \id_p$. 
For $\beta>\sqrt{p/n}$, the
principal eigenvector of $\bG$, and hence of $\hbSigma$ is positively
correlated with $\bv$, i.e.  $|\<\hbv_1(\hbSigma),\bv\>|$ is bounded away from zero. 
However, for $\beta<\sqrt{p/n}$, the noise component in $\hbSigma$ dominates 
and the two vectors become asymptotically
orthogonal, i.e. for instance $\lim_{n\to\infty}|\<\hbv_1(\hbSigma),\bv\>| =0$. 
In order to reduce the noise level, we must exploit the sparsity of
the spike $\bv$. 

Now, letting $\beta' \equiv \beta\|\bu\|^2/n\approx\beta$, and $\bw
\equiv \sqrt{\beta}\bZ^{\sT}\bu/n$, we can rewrite $\hbSigma$ as
\begin{align}
  \hbSigma &= \beta'\,\bv\bv^{\sT} + \bv\,\bw^{\sT}+\bw \, \bv^{\sT} + \frac{1}{n}\bZ^{\sT}\bZ\;\;
  - \id_p, . \label{eq:SigmaDef2}
\end{align}
For a moment, let us neglect the cross terms  $(\bv\bw^{\sT}+\bw
\bv^{\sT})$. The `signal' component $\beta'\,\bv\bv^{\sT}$ is sparse
with $s_0^2$ entries of magnitude $\beta'\theta^2/s_0$, which (in the regime of
interest $s_0 =\sqrt{n}/C$)  is equivalent to $C\theta^2 \beta/\sqrt{n}$. The `noise' component
$\bZ^{\sT}\bZ/n -\id_p$ is dense with entries of order $1/\sqrt{n}$.  
Assuming $s_0/\sqrt{n} < c$ for some small constant $c$, it should be possible to remove
most of the noise by thresholding the entries at level of order
$1/\sqrt{n}$. For technical reasons, we use the soft thresholding function  
$\eta:\reals\times\reals_{\ge 0}\to \reals, \, \eta(z; \tau) = \sign(z)(\abs{z}-\tau)_+$. 
We will omit the second argument from $\eta(\cdot; \cdot)$ wherever it is clear from context. 

      Consider again the decomposition \eqref{eq:SigmaDef2}. 
      Since the soft thresholding function $\eta(z; \tau/\sqrt{n})$ is affine 
  when $|z| \ge \tau/\sqrt{n}$, we would expect that the following decomposition
  holds approximately (for instance, in operator norm):
\begin{align} \label{eq:heurDecom}
  \eta(\hbSigma) &\approx \eta\left( \beta'\bv\bv^\sT \right) + \eta\left(  \frac{1}{n}\bZ^\sT\bZ -\id_p\right). 
\end{align}
Since $\beta' \approx \beta$ and each entry of $\bv\bv^\sT$ has magnitude at 
least $\theta^2/s_0$, the first term is still approximately rank one, with
\begin{align}
    \Big\lVert \eta\left( \beta' \bv\bv^\sT\right) - \beta \bv\bv^\sT \Big\rVert_{op}   &\le \frac{s_0\tau}{\sqrt{n}}. \label{eq:biasHeur}
\end{align}
This is straightforward to see since soft thresholding introduces a maximum bias of $\tau/\sqrt{n}$ per entry
of the matrix, while
the factor $s_0$ comes due to the support size of $\bv\bv^\sT$ (see Proposition \ref{prop:signal} below for a rigorous
argument). 

The main technical challenge now is to control the operator norm of
the perturbation $\eta(\bZ^\sT\bZ/n - \id_p)$. 
We know that $\eta(\bZ^\sT\bZ/n -\id_p)$ has entries of
variance $\delta(\tau)/n$, for $\delta(\tau) \approx \exp(-c\tau^2)$. 
If entries were independent with mild tail conditions, this would imply  --with high probability--
\begin{align}
  \bigg\lVert{\eta\bigg( \frac{1}{n}\bZ^\sT\bZ-\id_p \bigg)}\bigg\rVert_{op} %
  \le C\exp(-c\tau^2) \sqrt{\frac{p}{n}},\label{eq:kernRMnorm} 
\end{align}
for some constant $C$. Combining the bias bound from \myeqref{eq:biasHeur} and the
heuristic decomposition of \myeqref{eq:kernRMnorm} with the decomposition \eqref{eq:heurDecom} results in the bound
\begin{align}
    \Big\lVert \eta(\hbSigma) - \beta \bv\bv^\sT \Big\rVert_{op} &\le \frac{s_0\tau}{\sqrt{n}} + C\exp(-c\tau^2)\sqrt{\frac{p}{n}}.   
\end{align}
Our analysis formalizes this argument and shows that such a
bound is essentially correct when $p\le C\, n$. A modified bound is
proved for $p >C\, n$ (see Theorem \ref{thm:est_general} for a general statement).

The matrix $\eta\big(\bZ^\sT\bZ/n - \id_p\big)$ is a special
case of so-called inner-product kernel random matrices, which have
attracted 
recent interest within probability theory 
\cite{el2010information,el2010spectrum,cheng2013spectrum, fan2015spectral}. 
The basic object of study in this line of work is a matrix $\mathbf{M}\in \reals^{p\times p}$ of the type:
\begin{align}
    M_{ij} &= 
    f_n\bigg(\frac{\<\btz_i, \btz_j\>}{n} - \ind(i=j) \bigg)\,.\label{eq:kernelRMDef}
\end{align}
In other words, $f_n:\reals\to\reals$ is a kernel function and 
is applied entry-wise to the matrix $\bZ^\sT\bZ/n-\id_p$, with $\bZ$ a matrix
with independent standard normal entries as above and $\btz_i \in \reals^{n}$ are the columns
of $\bZ$. 

The key technical challenge in our proof is the analysis of the 
operator norm of such matrices, when $f_n$ is the
soft-thresholding function, with threshold of order $1/\sqrt{n}$. 
Earlier results 
are not general enough to cover this case.
  \cite{el2010information,el2010spectrum} provide conditions under which the 
  spectrum of $f_n(\bZ^\sT \bZ/n-\id_p)$
is close to a rescaling of the spectrum of $(\bZ^\sT \bZ/n-\id_p)$. 
We are interested  instead in a different regime in which the
spectrum of $f_n(\bZ^\sT \bZ/n-\id_p)$ is very different from the one of
$(\bZ^\sT \bZ/n-\id_p)$.
 \cite{cheng2013spectrum} consider $n$-dependent kernels, but their results are 
    asymptotic and concern the weak limit of the empirical spectral distribution of $f_n(\bZ^\sT\bZ/n-\id_p)$. 
    This does not yield an upper bound on the spectral norm
 of $f_n(\bZ^\sT\bZ/n-\id_p)$. Finally, \cite{fan2015spectral} consider the spectral norm of kernel random 
     matrices for smooth kernels $f$, only in the proportional regime $n/p\to  c\in (0,\infty)$.

Our approach to proving Theorem \ref{thm:est} follows instead the $\eps$-net method: we develop
high probability bounds on the maximum Rayleigh quotient:
\begin{align}
\max_{\by\in \sphere^{p-1}} \<\by, \eta(\bZ^\sT\bZ/n-\id_p)\by\> &= \max_{\by\in \sphere^{p-1}} \sum_{i, j }\eta\left( \frac{\<\btz_i, \btz_j\>}{n}; \frac{\tau}{\sqrt{n}} \right)y_i y_j,  
\end{align}
by discretizing $\sphere^{p-1} = \{\by\in\reals^p:\|\by\|=1\}$,  the unit sphere in $p$ dimensions.
For a fixed $\by$, the Rayleigh quotient $\<\by, \eta(\bZ^\sT\bZ/n -\id_p)\by\>$ is a (complicated)
function of the underlying Gaussian random variables $\bZ$. One might hope that it is Lipschitz
continuous with some Lipschitz constant $B = B(n, p, \tau, \by)$, thereby implying, by Gaussian
isoperimetry \cite{Ledoux}, that it concentrates to the scale $B$ 
around its expectation (i.e. 0). Then, by a standard union bound argument over a discretization
of the sphere, one would
obtain that the operator norm of $\eta\big(\bZ^\sT\bZ/n - \id_p\big)$ is typically no more than
$\sqrt{p}\sup_{\by\in \sphere^{p-1}} B(n, p, \tau, \by)$. 

Unfortunately, this turns
out not to be true over the whole space of $\bZ$, i.e. the Rayleigh
quotient is not Lipschitz continuous in the underlying Gaussian variables $\bZ$.
Our approach, instead, shows that for
\emph{typical} values of $\bZ$, we can control the gradient of $\<\by, \eta(\bZ^\sT\bZ/n - \id_p)\by\>$
with respect to $\bZ$, and extract the required concentration only from such local information
of the function. 
This is formalized in our concentration lemma \ref{lem:nonconvexConc}, 
which we apply extensively while proving Theorem \ref{thm:est}. This lemma
is a significantly improved version of the analogous result in \cite{deshpande2014sparse}.

\section{Practical aspects and empirical results}\label{sec:practical}

Specializing to the rank one case, Theorems \ref{thm:corr} and \ref{thm:main} show that
Covariance Thresholding succeeds with high probability for a number of
samples $n\gtrsim s_0^2$, while Diagonal Thresholding requires $n\gtrsim
s_0^2\log p$. The reader might wonder whether eliminating the $\log p$ factor
has any practical relevance or is a purely conceptual improvement.
Figure \ref{fig:supportRecovery} presents simulations on synthetic data under the
strictly sparse model, and the Covariance Thresholding algorithm of
Table \ref{alg:ct}, used in the proof of Theorem \ref{thm:main}. The
objective is to check whether the $\log p$ factor has an impact at
moderate $p$. We compare this with Diagonal Thresholding.
\begin{figure}[t]
  \includegraphics[width=0.3\linewidth]{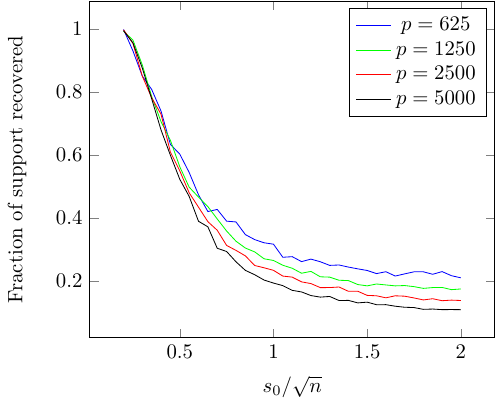}
  \includegraphics[width=0.3\linewidth]{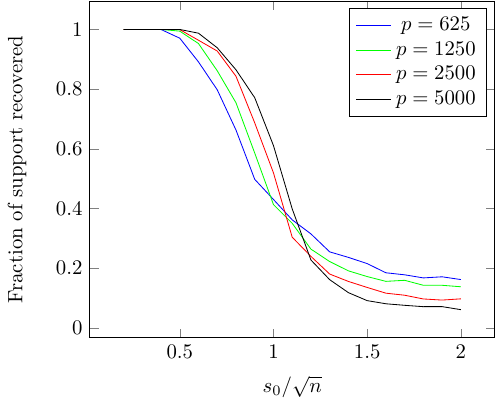}
  \includegraphics[width=0.3\linewidth]{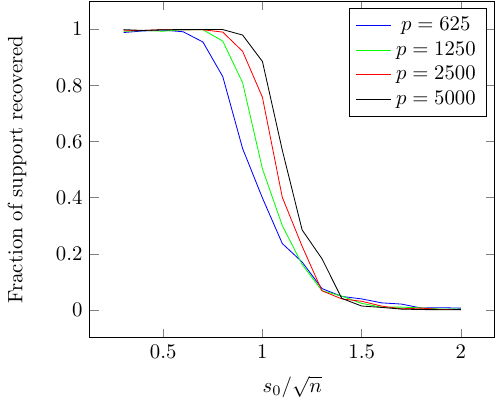}
  \caption{The support recovery phase transitions for Diagonal Thresholding (left) and
  Covariance Thresholding (center) and the data-driven version 
  of Section \ref{sec:practical} (right). For Covariance Thresholding, the
  fraction of support recovered correctly \emph{increases} monotonically with
  $p$, as long as  $s_0 \le c\sqrt{n}$ with $c\approx 1.1$. Further, it
  appears to converge to one throughout this region. For Diagonal
  Thresholding, the fraction of support recovered correctly
  \emph{decreases} monotonically with $p$ for all $s_0$ of order
  $\sqrt{n}$.
  This confirms that Covariance Thresholding (with or without knowledge
  of the support size $s_0$) succeeds with high
  probability for $s_0 \le c\sqrt{n}$, while Diagonal Thresholding
  requires a significantly sparser principal component. 
  \label{fig:supportRecovery}}
\end{figure}

We plot the empirical
success probability as a function of $s_0/\sqrt{n}$ for several values
of $p$, with $p=n$. The empirical success probability was computed by
using $100$ independent instances of the problem.
 A few observations are of interest: $(i)$ Covariance
Thresholding appears to have a significantly larger success
probability in the `difficult' regime where Diagonal Thresholding
starts to fail; $(ii)$ The curves for Diagonal
Thresholding appear to decrease monotonically with $p$ indicating that
$s_0$ proportional to $\sqrt{n}$  is not  the right scaling for this
algorithm  (as is
known from theory); $(iii)$ In contrast, the curves for Covariance
Thresholding become steeper for larger $p$, and, in particular, the
success probability increases with $p$ for $s_0\le 1.1\sqrt{n}$. This
indicates a sharp
threshold for $s_0 ={\rm const}\cdot\sqrt{n}$, as suggested by our theory.

In terms of practical applicability, our algorithm in Table \ref{alg:ct}
has the shortcomings of requiring knowledge of problem parameters
$s_0, \beta, \theta$. Furthermore, the thresholds $\rho, \tau$ suggested by theory need
not be optimal. We next describe a 
principled approach to estimating (where possible) the parameters
of interest and running the algorithm in a purely data-dependent manner.
Assume the following model, for $i\in [n]$
\begin{align*}
  \bx_i &= \bmu + \sum_q\sqrt{\beta_q}u_{q,i}\bv_q + \sigma\bz_i, 
\end{align*}
where $\bmu\in\reals^p$ is a fixed mean vector, $u_{q, i}$ have mean $0$
and variance $1$, and $\bz_i$ have mean $0$ and covariance $\id_p$. 
Note that our focus in this section is not on rigorous analysis, but instead to demonstrate
a principled approach to applying covariance thresholding in practice.
We proceed as follows:

\begin{description}
  \item [Estimating $\bmu$, $\sigma$:] 
    We let $\widehat\bmu = \sum_{i=1}^n \bx_i/n$ be the empirical mean estimate
    for $\mu$. Further letting $\obX=\bX-\mathbf{1}\widehat{\bmu}^\sT$ we 
    see that $pn-(\sum_q k_q)n \approx pn$ entries of $\obX$ are mean $0$ and variance $\sigma^2$. 
    We let $\widehat{\sigma} = {\MAD(\obX)}/{\nu}$ 
where $\MAD(\,\cdot\,)$ denotes the median absolute
    deviation of the entries of the  matrix in the argument, and $\nu$
    is a constant scale factor. Guided by  the
    Gaussian case, we  take $\nu = \Phi^{-1}(3/4) \approx 0.6745$.

  \item[Choosing $\tau$:]
    Although in the statement of the theorem, our choice of $\tau$ depends on the SNR $\beta/\sigma^2$,
    it is reasonable to instead
    threshold `at the noise level', as follows. The noise component of
    entry $i,j$ of the sample covariance (ignoring lower order
    terms) is given by $\sigma^2\<\bz_i, \bz_j\>/n$. By the central limit theorem, 
    $\<\bz_i, \bz_j\>/\sqrt{n} \convD \normal(0, 1)$. Consequently, $\sigma^2\<\bz_i, \bz_j\>/n \approx
    \normal(0, \sigma^4/n)$, and we need to choose the (rescaled) threshold proportional
    to $\sqrt{\sigma^4} = \sigma^2$. Using previous estimates, we let 
$\tau = \nu'\cdot \widehat{\sigma}^2$
    for a constant $\nu'$. In simulations, a choice $3\lesssim \nu' \lesssim 4$ appears to work well.

  \item[Estimating $r$:] 
      We  define $\hbSigma = \obX^\sT\obX/n-\widehat{\sigma}^2\id_p$ and soft threshold it
    to get $\eta(\hbSigma)$ using $\tau$ as above.
    Our proof of Theorem \ref{thm:corr} relies on the fact
    that $\eta(\hbSigma)$ has $r$ eigenvalues that are separated from the 
    bulk of the spectrum. Hence, we estimate $r$ using $\widehat{r}$: the number of
    eigenvalues separated from the bulk in $\eta(\hbSigma)$. The  edge of the
    spectrum can be computed numerically using the Stieltjes transform method as
    in \cite{cheng2013spectrum}. 
  \item[Estimating $\bv_q$:]
    Let $\hbv_q$ denote the $q^{\text{th}}$
    eigenvector of $\eta(\hbSigma)$. Our theoretical analysis indicates that $\hbv_q$ is
    expected to be close to $\bv_q$. In order to denoise $\hbv_q$, we assume
    $\hbSigma\hbv_q\approx (1-\delta)\bv_q + \beps_q$, where $\beps_q$ is additive random noise
    (perhaps with some sparse corruptions). 
    We then threshold $\hbSigma\bv_q$ `at the
    noise level' to recover a better estimate of $\bv_q$. To do this, we
    estimate the standard deviation of the ``noise'' $\beps$ by
$\widehat{\sigma_{\beps}}
    = {\MAD(\hbv_q)}/{\nu}$. Here we set --again guided by the Gaussian heuristic--  $\nu \approx 0.6745$. Since $\bv_q$ is sparse,
    this procedure returns a good estimate for the size of the noise
    deviation. We let $\hbv'_q$ denote the vector obtained by hard 
    thresholding $\hbv_q$: set  $\hbv'_i = \hbv_{q,i} \text{ if }
    \abs{\hv_{q,i}} \ge \nu' \widehat{\sigma}_{\beps_q}$
    and $
	0 \text{ otherwise.}$
    We then let $\hbv^*_q = \hbv'_q/\norm{\hbv'_q}$ and return $\hbv^*_q$
    as our estimate for $\bv_q$.

\end{description}
Note that --while different in several respects-- this empirical approach
shares the same philosophy of the algorithm in Table
\ref{alg:ct}.
On the other hand, the data-driven algorithm presented in this section  is less
straightforward to analyze, a task that we defer to future work.

Figure \ref{fig:supportRecovery} also shows results of a support recovery
experiment using the `data-driven' version of this section. Covariance thresholding 
in this form also appears to work for supports of size $s_0 \le \text{const}\sqrt{n}$. 
Figure \ref{fig:threePeak} shows the performance of vanilla PCA, Diagonal Thresholding
and Covariance Thresholding on the ``Three Peak'' example of \cite{johnstone2004sparse}. 
This signal is sparse in the wavelet domain and the simulations employ the data-driven
version of covariance thresholding. A similar experiment with the ``box'' example of Johnstone
and Lu is provided in Figure \ref{fig:block}. These experiments demonstrate that, while for  large values of $n$ both Diagonal
Thresholding and Covariance Thresholding perform well, the latter
appears superior for smaller values of $n$.

\begin{figure}[h] 
  \centering
  \includegraphics[scale=0.5]{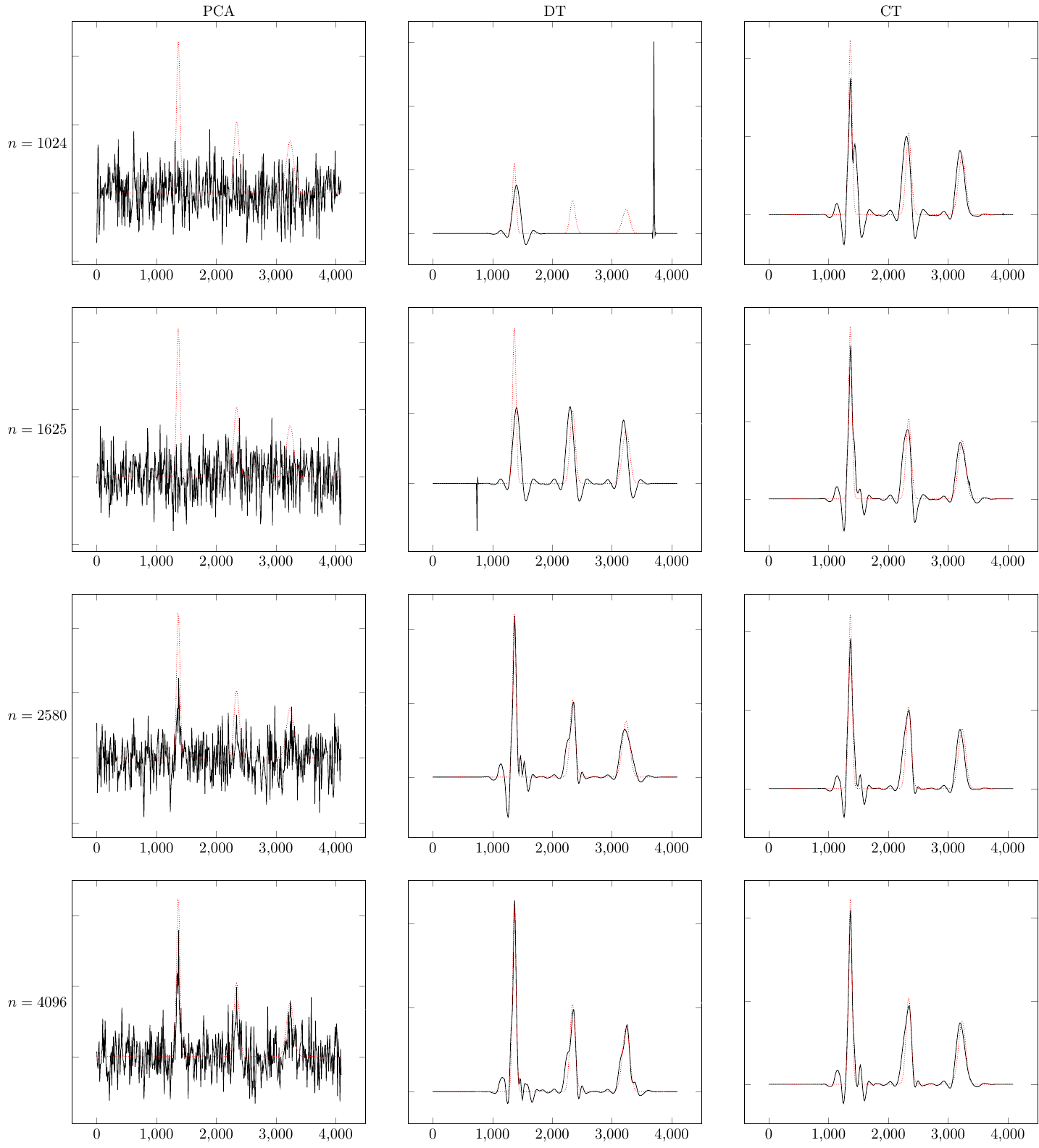}
  \caption{The results of Simple PCA, Diagonal Thresholding and Covariance Thresholding (respectively) for
    the ``Three Peak'' example of \cite{johnstone2009consistency} (see Figure 1
  of the paper). The signal is sparse in the `Symmlet 8' basis. We use $\beta = 1.4, p=4096$, and the rows correspond to sample sizes $n=1024, 1625, 2580, 4096$
  respectively. Parameters for Covariance Thresholding are chosen as in Section \ref{sec:practical},
  with $\nu' = 4.5$. Parameters for Diagonal Thresholding are from \cite{johnstone2009consistency}. On each curve,
  we superpose the clean signal (dotted).  \label{fig:threePeak}}
\end{figure}

\begin{figure}[h]
  \centering
  \includegraphics[scale=0.5]{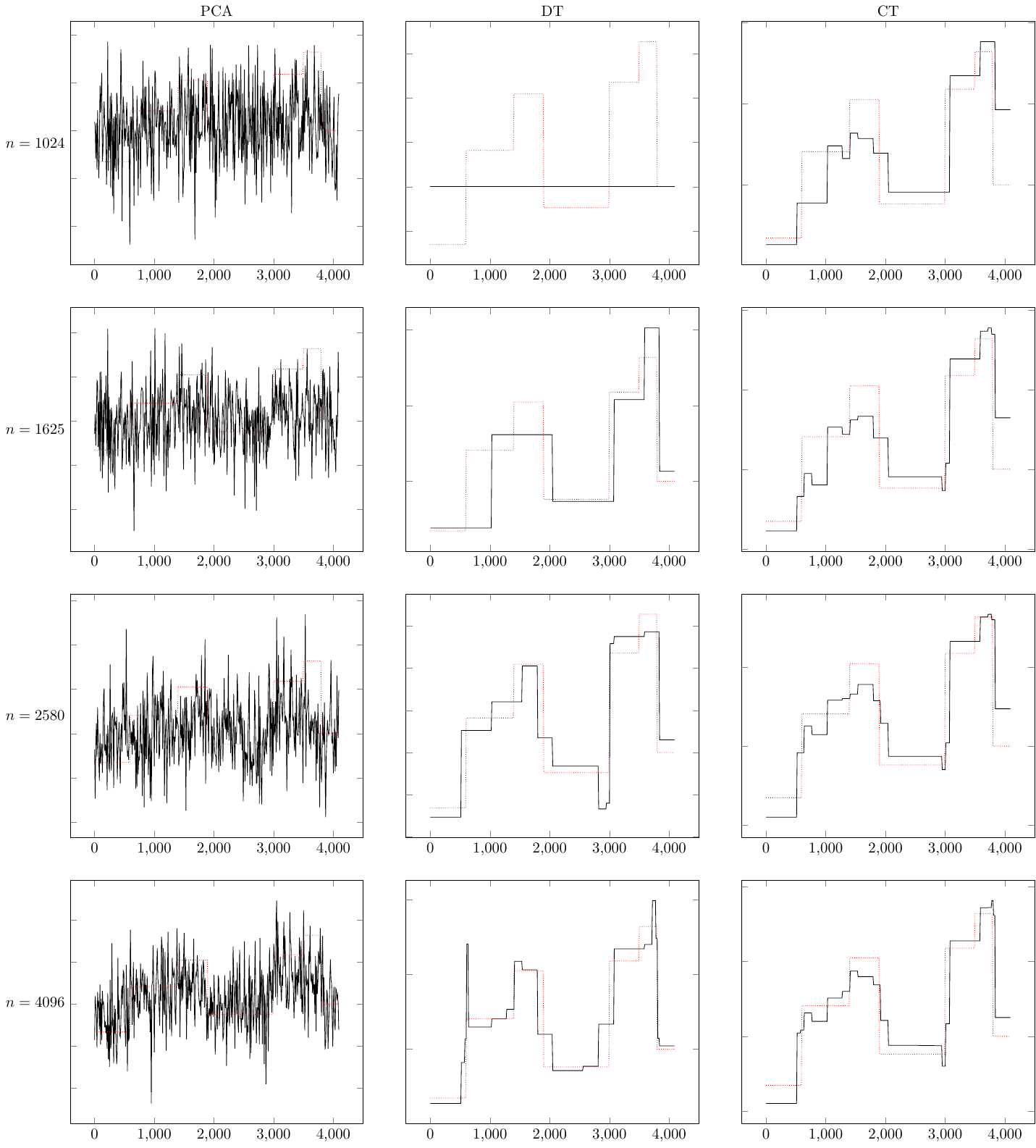}
  \caption{The results of Simple PCA, Diagonal Thresholding and Covariance Thresholding (respectively) for
    a synthetic block-constant function (which is sparse in the Haar wavelet basis). 
    We use $\beta = 1.4, p=4096$, and the rows correspond to sample sizes $n=1024, 1625, 2580, 4096$
  respectively. Parameters for Covariance Thresholding are chosen as in Section \ref{sec:practical},
  with $\nu' = 4.5$. Parameters for Diagonal Thresholding are from \cite{johnstone2009consistency}. On each curve,
  we superpose the clean signal (dotted).
}
  \label{fig:block}
\end{figure}

\section{Proof preliminaries}\label{sec:prelim}
In this section we review some notation and preliminary facts
that we will use throughout the paper. 

\subsection{Notation}

We let $[m] = \{1,2,\dots,m\}$ denote the set of first $m$ integers.
We will represent vectors using boldface lower case letters, 
e.g. $\bu, \bv, \bx$. The entries of a vector $\bu\in\reals^n$ will be
represented by $u_i, i\in[n]$.
Matrices are represented using boldface upper
case letters e.g. $\bA, \bX$. The entries of a matrix $\bA \in\reals^{m\times n}$ are 
represented by $\bA_{ij}$ for $i\in[m], j\in[n]$. 
Given a matrix $\bA\in\reals^{m\times n}$, we generically let $\ba_1$,
$\ba_2, \dots, \ba_m$ denote its rows, and $\bta_1$,
$\bta_2, \dots, \bta_n$ its columns. 

For $E\subseteq [m]\times [n]$, we define the projector operator
$\proj_E:\reals^{m\times n}\to \reals^{m\times n}$ by letting
$\proj_E(\bA)$ be the matrix with entries
\begin{align}
\proj_{E}(\bA)_{ij} = \begin{cases}
\bA_{ij} & \mbox{if $(i,j)\in E$,}\\
0 & \mbox{otherwise.}
\end{cases}
\end{align}
For a matrix $\bA\in\reals^{m\times n}$, and a set $E\subseteq[n]$, we
define its column restriction $\bA_{E}\in\reals^{m\times n}$ to be the
matrix obtained by setting to $0$ columns outside $E$:
\begin{align*}
  (\bA_{E})_{ij} &= \begin{cases}
    \bA_{ij} &\text{ if }j\in E,\\
    0 &\text{otherwise. }
  \end{cases}
\end{align*}
Similarly $\by_E$ is obtained from $\by$ by setting to zero all
indices outside $E$.
The operator norm   of a matrix $\bA$ is denoted by $\norm{\bA}$ (or
$\norm{\bA}_{op}$)
and its Frobenius norm by $\norm{\bA}_F$. We write $\norm{\bx}$ for the standard $\ell_2$ 
norm of a vector $\bx$. Other vector norms such as $\ell_1$ or $\ell_\infty$ are 
denoted with appropriate subscripts. 

We let $\sC_q$ denotes the support of the $q^{\text{th}}$
spike $\bv_q$. Also, we denote the union of the supports of $\bv_q$ by $\sC=\cup_q\sC_q$. 
The complement of a set $E\in[n]$ is denoted by $E^c$. 

We write $\eta(\cdot; \cdot)$ for the soft-thresholding function.
By $\der\eta(\cdot ; \tau)$ we denote the derivative of $\eta(\cdot; \tau)$ with respect
to the \emph{first} argument, which exists Lebesgue almost everywhere. To simplify the
notation, we omit
the second argument when it is understood from context. 

For a random variable $Z$ and a measurable set $\cA$  we write $\E\{Z; \cA\}$ to denote $\E\{Z\ind(Z\in\cA)\}$,
the expectation of $Z$ constrained to the event $\cA$. 

In the statements of our results, consider the limit of large $p$ and large
$n$ with certain conditions on $p, n$ (as in Theorem \ref{thm:est}). This limit will be referred to either as ``$n$ large enough''
or ``$p$ large enough'' where the phrase ``large enough'' indicates dependence of $p$ (and
thereby $n$) on specific problem parameters.

The Gaussian distribution function will be denoted by $\Phi(x)  =\int_{-\infty}^x e^{-t^2/2}\, \de t/\sqrt{2\pi}$.

\subsection{Preliminary facts}

Let $\sphere^{N-1}$ denote the unit sphere in $N$ dimensions, i.e. $\sphere^{N-1} = \{\bx\in\reals^N : \norm{\bx} = 1\}$.
We use the following definition \cite[Definition 5.2]{Vershynin-CS} of
the $\eps$-net of a set $X\subseteq\reals^n$:
\begin{definition}[Nets, Covering numbers]\label{def:nets}
  A subset $T^\eps(X)\subseteq X$ is called an $\eps$-net of $X$ if every point in
  $X$ may be approximated by one in $T^\eps(X)$ with error at most $\eps$. More precisely:
  \begin{align*}
    \forall x\in X,\quad \inf_{y\in T^\eps(X)} \norm{x - y} &\le \eps.
  \end{align*}
  The minimum cardinality of an $\eps$-net of $X$, if finite,  is called its covering number.
\end{definition}

The following two facts are useful while using $\eps$-nets to bound the
spectral norm of a matrix. For proofs, we refer the reader to \cite[Lemmas 5.2, 5.4]{Vershynin-CS}.
\begin{lemma}
  Let $S^{n-1}$ be the unit sphere in $n$ dimensions. Then there exists an $\eps$-net
  of $S^{n-1}$, $T^\eps(S^{n-1})$ satisfying:
  \begin{align*}
    |T^\eps(S^{n-1})| \le \left( 1+ \frac{2}{\eps} \right)^n.
  \end{align*}
  \label{lem:epsnetcard}
\end{lemma}

\begin{lemma}
    Let $\bA\in\reals^{n\times n}$ be a symmetric matrix. Then, there exists $\bx \in T^{\eps}(S^{n-1})$ such
    that
    \begin{align}
        \abs{ \<\bx, \bA\bx\>} &\ge (1-2\eps)\norm{\b{A}}.
    \end{align}
  \label{lem:specnormbnd}
\end{lemma}
\begin{proof}
    Firstly, we have $\norm{A} = \max_{\bx\in \sphere^{n-1}} \abs{\<\bx, \bA\bx\>} = \max_{\bx\in \sphere^{n-1}} \norm{\bA\bx}$. Let $\bx_*$ be the
    maximizer (which exists as $\sphere^{n-1}$ is compact and $|\<\bx, \bA\bx\>|$ is continuous in $\bx$). 
    Choose $\bx\in T^\eps_n$ so that $\norm{\bx - \bx_*} \le \eps$. Then:
    \begin{align}
        \<\bx, \bA\bx\> 
        &= \<\bx - \bx_*, \bA (\bx+\bx_*)\> + \<\bx_*, \bA\bx_*\> \, .
    \end{align}
    The lemma then follows as $\abs{\<\bx, \bA(\bx - \bx_*)\>} \le \norm{\bx+\bx_*}\norm{\bA}\norm{\bx-\bx_*}\le 2\eps\norm{\bA}$. 
\end{proof}

Throughout the paper we will denote by $T^\eps_N$ an $\eps$-net on the unit sphere
$\sphere^{N-1}$ that satisfies Lemma \ref{lem:epsnetcard}. For a subset of indices $\sS\subset [N]$ we denote
by $T^\eps_N(\sS)$ the natural isometric embedding of $T^\eps_{\sS}$ in $\sphere^{N-1}$.

We now state a general concentration lemma. This will be our basic tool to 
establish Theorem \ref{thm:corr}, and thereby Theorem \ref{thm:main}.

\begin{lemma} 
    \label{lem:nonconvexConc}
    Let $\bz \sim \normal(0, \id_N)$ be vector of $N$ i.i.d. standard normal
    variables. Suppose $S$ is a finite set and we have functions $F_s: \reals^N\to \reals$ for
    every $s\in S$. Assume $\cG \in \reals^N\times \reals^N$ is a Borel set such that for Lebesgue-almost every
    $(\bx, \by) \in \cG$:
    \begin{align}
        \max_{s\in S} \max_{t \in [0, 1]} \norm{\grad F_s (\sqrt{t}\bx + \sqrt{1-t}\by)} &\le L \, .
        \label{eq:gradCond}
    \end{align}
    Then,  for any $\Delta > 0$:
    \begin{align}
        \P\Big\{ \max_{s\in S} \abs{F_s(\bz) - \E F_s(\bz)} \ge \Delta \Big\} &\le C\abs{S}\,\exp\Big(-\frac{\Delta^2}{CL^2}\Big)  
         + \frac{C}{\Delta^2}  \E\Big\{\max_{s\in S} \big[(F_s(\bz) -  F_s(\bz'))^2 \big]; \cG^c \Big\}.
        \label{eq:concLem}
    \end{align}
    Here $\bz'$ is an independent copy of $\bz$. 
\end{lemma}
\begin{proof}
    We use the Maurey-Pisier method along with symmetrization. 
    By centering, assume that $\E F_s(\bz) = 0$ for all $s\in S$. Further, by
    including the functions $-F_s$ in the set $S$ (at most doubling its size), it suffices to prove the 
    one-sided version of the inequality:
    \begin{align}
        \P\{\max_{s\in S} F_s(\bz) \ge \Delta\} &\le 
        C\abs{S}\,\exp\Big(-\frac{\Delta^2}{CL^2}\Big) + \frac{C}{ \Delta^{2}}  \E\{\max_s (F_s(\bz)- F_s(\bz'))^2;  \cG^c\}\, .
        \label{eq:oneSidedDef}
    \end{align}

    We first implement the symmetrization. Note that:
    \begin{align}
        \{\bx: \max_s F_s(\bx) \ge \Delta\} &\subseteq \{\bx: \max_{x\in \reals, s\in S} [2x F_s(\bx) -x^2] \ge \Delta^2\} \\
        \{\bx, \by: \max_s [F_s(\bx) - F_s(\by)] \ge \Delta\} &\subseteq \{\bx, \by: \max_{x\in \reals, s\in S} [2x (F_s(\bx) - F_s(\by))  -x^2]\ge \Delta^2\}.
    \end{align}
    Furthermore, by centering, $F_s(\bz) = \E\{F_s(\bz) - F_s(\bz')|\bz\}$. Hence for any non-decreasing convex function
    $\phi(z)$:
    \begin{align}
        \E\Big\{ \phi\big(\max_{x, s} [2x F_s(\bz) - x^2] \big)\Big\} &\le \E\bigg\{\phi\Big(\max_{x, s} \big[\E\{2x F_s(\bz) - 2x F_s(\bz') - x^2  | \bz\} \big]\Big) \bigg\} \\
        &\stackrel{(a)}{\le} \E\bigg\{ \phi\Big(\E\big\{\max_{x, s} [2x(F_s(\bz) - F_s(\bz')) - x^2] |\bz \big\} \Big) \bigg\} \\
        &\stackrel{(b)}{\le} \E\Big\{ \phi\big( \max_{x, s} [2x(F_s(\bz) - F_s(\bz')) - x^2] \big) \Big\}. \label{eq:symm}
    \end{align}
    Here we use Jensen's inequality with the monotonicity of $\phi(\cdot)$ to obtain $(a)$ and with
    the convexity of $\phi(\cdot)$ to obtain $(b)$.

    Now we choose $\phi(z) = (z-a)_+$, for $a = \Delta^2 /2$. 
    \begin{align}
        \P\{\max_{s} F_s(\bz) \ge \Delta\} &\le \P\big\{\max_{x, s} [2x F_s(\bz) - x^2] \ge \Delta^2\big\} \\
        &\stackrel{(a)}{\le} \phi(\Delta^2)^{-1} \E\Big\{ \phi\big(\max_{x, s}[2x F_s(\bz) - x^2]\big)\Big\} \\
        &\stackrel{(b)}{\le} \phi(\Delta^2)^{-1} \E\Big\{ \phi\big(\max_{x, s} [2x (F_s(\bz) -  F_s(\bz')) - x^2]\big) \Big\} \\
        &= \phi(\Delta^2)^{-1} \E\Big\{ \phi\big( \max_s [(F_s(\bz) - F_s(\bz'))^2]\big) \Big\} \\
        &= \phi(\Delta^{2})^{-1} \bigg(\E\Big\{ \phi\big( \max_s [(F_s(\bz) - F_s(\bz'))^2 ]\big) ; \cG \Big\} \nonumber \\ 
        &\quad +  \E\Big(\big\{\phi(\max_s  [(F_s(\bz) - F_s(\bz'))^2] ; \cG^c \big\}\Big)\, . \label{eq:goodBadDecomp}
    \end{align}
    Here $(a)$ is Markov's inequality, and $(b)$ is the symmetrization bound \myeqref{eq:symm}, where
    we use the fact that $\phi(z) = (z-a)_+$ is non-decreasing and convex in $z$. 
    
    At this point, it is
    easy to see that the lemma follows if we are able to control the first term in \myeqref{eq:goodBadDecomp}. 
    We establish this 
    via the Maurey-Pisier method. Define the path $\bz(\theta) \equiv \bz \sin\theta + \bz'\cos \theta$,
    the velocity $\bzdot \equiv \d\bz/\d\theta = \bz \cos\theta - \bz'\sin\theta$.
    \begin{align}
        \E\Big\{ \phi\big( \max_s [(F_s(\bz) - F_s(\bz'))^2 ]\big) ; \cG \Big\}
        &= \int_{0}^{\infty} \P\Big\{ \big(\max_s[(F_s (\bz) - F_s(\bz'))^2] - a\big)_+ \ind(\cG) \ge x  \Big\} \d x\\
        &= \int_{0}^{\infty} \P\big\{ \max_s [\abs{F_s(\bz) - F_s(\bz')}] \ge \sqrt{x+a} ;  \cG \big\} \d x \\
        &\le 2|S| \int_{a}^{\infty} e^{-\lambda\sqrt{x}}\max_s\Big[\E\big\{\exp\{\lambda (F_s(\bz) - F_s(\bz'))\} ; \cG   \big\}\Big]\d x \, , \label{eq:uptoExpMom}
    \end{align}
    where, in the last inequality we use the union bound followed by Markov's inequality. To 
    control the exponential moment, note that $F_s(\bz) - F_s(\bz') = \int_{0}^{\pi/2} \<\grad F(\bz(\theta)), \bzdot(\theta)\> \d \theta$
    whence, using Jensen's inequality:
    \begin{align}
        \E\Big\{ \exp\big\{\lambda (F_s(\bz) - F_s(\bz'))\big\} ; \cG\Big\} &= \E\bigg\{\exp\Big(  \int_{0}^{\pi/2} \lambda\<\grad F_s (\bz (\theta)), \bzdot(\theta)\> \d \theta\Big) ; \cG \bigg\}  \\
        &\le \frac{2}{\pi}\int_{0}^{\pi/2} \E \Big\{ \exp \big( \lambda \pi\< \grad F_s(\bz(\theta)), \bzdot(\theta) \> /2\big) ; \cG \Big\} \d\theta.
        \label{eq:expmomcontrol}
    \end{align}
    Define the set $\cG_\theta = \{ (\bz, \bz'): \max_s \norm{\grad F_s (\bz(\theta))} \le L\}$. Then:
    \begin{align}
        \E\Big\{ \exp\big\{\lambda (F_s(\bz) - F_s(\bz'))\big\} ; \cG\Big\} & %
\stackrel{(a)}{\le}\frac{2}{\pi}\int_{0}^{\pi/2} \E \Big\{ \exp \big( \lambda \pi\< \grad F_s(\bz(\theta)), \bzdot(\theta) \> /2\big) ; \cG_\theta\Big\} \d \theta \\
&\stackrel{(b)}{=}  \frac{2}{\pi} \int_{0}^{\pi/2} \E\bigg\{ \exp \Big(  \frac{\lambda^2\pi^2 \norm{\grad F_s(\bz(\theta))}^2}{8} ; \cG_\theta \Big)\bigg\}\d\theta \\
&\stackrel{(c)}{\le} \exp \Big(  \frac{\lambda^2 \pi^2 L^2}{8}\Big).%
        \label{eq:expmomcontrol2}
    \end{align}
    Here $(a)$ follows as $\cG_\theta\supseteq\cG$. Equality $(b)$ follows from noting that $\cG_\theta$ 
    is measurable with respect to $\bz(\theta)$ and, hence, first integrating with respect to $\bzdot(\theta) = \bz\cos\theta - \bz'\sin\theta$, a Gaussian
    random variable that is independent of $\bz(\theta)$. The final inequality $(c)$ follows by using the fact that $\norm{\grad F_s(\bz(\theta))} \le L$ on the 
    set $\cG_\theta$. 

    Since this bound is uniform over $s\in S$, we can use it in \eqref{eq:uptoExpMom}:
    \begin{align}
        \E\left\{ \phi( \max_s (F_s(\bz) - F_s(\bz'))^2) ; \cG \right\} %
        &\le 2|S| \int_{a}^{\infty} \exp \Big(    -\lambda \sqrt{x} + \frac{\lambda^2\pi^2 L^2}{8}\Big) \d x \\
        &\le \frac{4|S|}{\lambda^2} (1+\lambda\sqrt{a}) \exp\Big(   - \lambda \sqrt{a} + \frac{\lambda^2\pi^2L^2}{8}  \Big)
    \end{align}
    We can now set $\lambda = 4\sqrt{a}/\pi^2 L^2$, to obtain the exponent above as $ - 2 a/ \pi^2 L^2 = -\Delta^2/\pi^2L^2$. 
    The prefactor $(1+\lambda\sqrt{a})\lambda^{-2}$ is bounded by $CL^2\max(, L^2/\Delta^2)$ when $a = \Delta^2/2$. Therefore, as
    required, we obtain:
    \begin{align}
        \E\left\{ \phi( \max_s (F_s(\bz) - F_s(\bz'))^2) ; \cG \right\} %
    &\le C\max(1, L^4/\Delta^4) \exp\Big(-\frac{\Delta^2}{CL^2}\Big)%
        \label{eq:goodSetBound}
    \end{align}
    Combining this with \myeqref{eq:goodBadDecomp} and the fact that $\phi(\Delta^2)^{-1} \le C\Delta^{-2}$ gives \myeqref{eq:oneSidedDef}
    and, consequently, the lemma. 
\end{proof}

By a simple application of Cauchy-Schwarz, this lemma implies the following.
\begin{corollary}\label{corr:nonconvexCS}
    Under the same conditions as Lemma \ref{lem:nonconvexConc},
    \begin{align}
        \P\Big\{ \max_{s\in S} \abs{F_s(\bz) - \E F_s(\bz)} \ge \Delta \Big\} &\le C\abs{S}\,\exp\Big(-\frac{\Delta^2}{CL^2}\Big)  \nonumber\\ %
        &\quad + \frac{C}{\Delta^2}  \E\Big\{\max_{s\in S} \big[(F_s(\bz) -  F_s(\bz'))^4 \big]\Big\}^{1/2} \P\{\cG^c \}^{1/2}.%
    \end{align}
\end{corollary}
The following two lemmas are well-known concentration of measure results. 
The forms below can be found in \cite[Corollary 5.35]{Vershynin-CS}, \cite[Lemma 1]{laurent2000adaptive}
respectively.
\begin{lemma}\label{lemma:NormWishart}
  Let $\bA\in\reals^{M\times N}$ be a matrix with i.i.d. standard normal
  entries, i.e. $\bA_{ij}\sim\normal(0, 1)$. Then, for every $t\ge 0$:
  \begin{align}
      \P\left\{ \norm{\bA}_{op} \ge \sqrt{M} + \sqrt{N} + t \right\} &\le \exp\left( -\frac{t^2}{2} \right).
  \end{align}
  \label{lem:gaussianmatnorm}
\end{lemma}

\begin{lemma}
    Let $\bz\sim\normal(0, \id_N)$. Then
    \begin{align}
        \P\{ \norm{\bz}^2 \ge N + 2\sqrt{N t} + 2t\} &\le \exp(-t).\label{eq:chisquare}
    \end{align}
    \label{lem:chisquare}
\end{lemma}

\section{Proof of Theorem \ref{thm:est}}\label{sec:proofcorr}

Since $\hbSigma = \bX^{\sT}\bX/n - \id_p$, we have: 
\begin{align}\label{eq:empcov}
  \hbSigma &= \sum_{q=1}^r \left\{\frac{\beta_q\norm{\bu_q}^2}{n}\bv_q(\bv_q)^\sT + \frac{\sqrt{\beta_q}}{n}
\big(\bv_q(\bZ^{\sT}\bu_q)^\sT+ (\bZ^\sT\bu_q)\bv_q^\sT\big)\right\}  \nonumber\\
  &\quad+ \sum_{q\ne q'} \left\{\frac{\sqrt{\beta_q\beta_{q'}}\<\bu_q,
  \bu_{q'}\>}{n}\bv_q(\bv_{q'})^\sT \right\} + \frac{\bZ^\sT\bZ}{n} - \id_p\, .
\end{align}
We let $\sD =\{(i, i) : i\in[p]\setminus \sQ\}$ be the diagonal
entries not included in any support. (Recall that $\sC=\cup_q\sC_q$ denote the
union of the supports.) Further let  $\sE = \sC\times\sC$, $\sF = (\sC^c\times\sC^c)\bsl\sD$,
and $\sG = [p]\times[p]\bsl(\sD\cup\sE\cup\sF)$, or, equivalently $\sG = (\sC\times \sC^c)\cup(\sC^c\times \sC)$. 
Since these are disjoint we have:
\begin{align}
  \eta(\hbSigma) &=  \underbrace{\proj_{\sE}\left\{ \eta(\hbSigma) \right\}}_{\bS}  %
  +\underbrace{\proj_{\sF}\left\{ \eta\left( \hbSigma \right) \right\}}_{\bN}  %
  +\underbrace{\proj_{\sG}\left\{ \eta(\hbSigma) \right\}}_{\bC}%
  +\underbrace{\proj_{\sD}\left\{ \eta( \hbSigma) \right\}}_{\bD}. \label{eq:decomp}
\end{align}
The first term corresponds to the `signal' component, while
the last three terms correspond to the `noise' component. 

Theorem
\ref{thm:est} is a direct consequence of the next five propositions.
The first demonstrates that, even for a low level
of thresholding, viz. $\tau <\sqrt{\log p}/2$, 
the term $\b{N}$ has small operator norm.
The second demonstrates that the soft thresholding operation preserves
the signal in the term $\b{S}$. The next two propositions show that the
cross and diagonal terms $\b{C}$ and $\b{D}$ are negligible as well. 
Finally, in the last proposition, we demonstrate that, 
for the regime of thresholding far above the noise level, i.e. $\tau>C\sqrt{\log p}$, 
the noise terms $\b{N}$ and $\b{C}$ vanish entirely.

\begin{proposition} \label{prop:noise}
  Let $\bN$ denote the second term
  of \myeqref{eq:decomp}. Since $\sF = \sC^c \times \sC^c \bsl \sD$,  
  \begin{align}
      \bN &= \proj_{\sF} \left( \eta(\hbSigma) \right) = \proj_{\sF}\left\{\eta\left( \frac{1}{n}\bZ^\sT \bZ \right)\right\}.
  \end{align}
  Then, there exists an absolute constant $C$ such that
  the following happens. Assuming that $(i)$ $\tau < \sqrt{\log p}/2$ and $(ii)$ $n > C\log p$,
  then with  probability $1-o(1)$
\begin{align}
   \norm{\bN}_{op} \le C\,  \left(\sqrt{\frac{p}{n}} \vee \frac{p}{n}\right)\, e^{-\tau^2/C}\, .
\end{align}
\end{proposition}
\begin{proposition}\label{prop:signal}
  Let $\bS$ denote the  first term in 
  \myeqref{eq:decomp}: 
  \begin{align}
    \bS &= \proj_\sE\left\{\eta(\hbSigma)\right\}.
  \end{align}
  Assume that $(i)$ $s_0/n < 1$ and $(ii) n > C\log p$:
  Then with probability $1-o(1)$:
\begin{align}
\big\| \bS-\bSigma\big\|_{op}\le\frac{2\tau s_0}{\sqrt{n}} +C(\beta\vee 1)\sqrt{\frac{s_0}{n}}\, .
\end{align}
\end{proposition}
\begin{proposition}\label{prop:cross}
    Let $\b{C}$ denote the matrix corresponding to the third term
  of \myeqref{eq:decomp}:
  \begin{align*}
      {\b{C}} &= \proj_{\sG}\left\{\eta(\hbSigma)\right\}.
  \end{align*}
  Assuming the conditions of Proposition \ref{prop:noise} and, additionally, 
  that $s_0^2 \le p$, there exist constants $C, c$ such that with probability $1-o(1)$
\begin{align}
    \norm{\b{C}}_{op} &\le C\,\tau e^{-c\tau^2/(\beta\vee 1)} \sqrt{\frac{p}{n}}\vee \frac{p}{n}.  
\end{align}
\end{proposition}

\begin{proposition}\label{prop:diag}
    Let $\b{D}$ denote the matrix corresponding to the third term
  of \myeqref{eq:decomp}:
  \begin{align*}
      \b{D}&= \proj_{\sD}\left\{\eta(\hbSigma)\right\}.
  \end{align*}
  With probability $1-o(1)$ we have that $\norm{\b{D}}_{op} \le C\sqrt{n^{-1}\log p}$.
\end{proposition}

\begin{proposition}
    \label{prop:largeThresh}
    For some absolute constant $C_0$, we have for $\tau \ge C_0 (\beta \vee 1) \sqrt{\log p}$
    that, with probability $1-o(1)$:
    \begin{align}
        \forall i, j\quad N_{ij} = C_{ij} = 0.
    \end{align}
    Therefore, $\norm{\b{N}}_{op} = 0$ and $\norm{\b{C}}_{op}= 0$.
\end{proposition}

\begin{remark}
    At this point we remark that the probability $1-o(1)$ can be made
    quantitative, for e.g. of the form $1-\exp(-\min(\sqrt{p}, n)/C_1)$, for
every $n$ large enough. For simplicity of exposition we do not pursue this 
in the paper. 
\end{remark}

We defer the proofs of Propositions \ref{prop:noise},
\ref{prop:signal}, \ref{prop:cross}, \ref {prop:diag} and \ref{prop:largeThresh} to
Sections \ref{subsec:proofnoise}, \ref{subsec:proofsignal},
\ref{subsec:proofcross},  \ref{subsec:proofdiag} and \ref{subsec:prooflargethresh}
respectively.
By combining them for $\beta=O(1)$,  we immediately obtain the following bound.
\begin{theorem}
    \label{thm:est_general}
There exist numerical constants $C_0,C_1$ such that the following happens.
   Assume $\beta\le C_0$, $n>C_1\log p$    and $\tau \le \sqrt{\log p}/2$. Then 
  with probability $1-o(1)$:
    \begin{align}
\big\|\eta(\hbSigma) - \bSigma \big\|_{op}&\le \frac{2\tau s_0}{\sqrt{n}} +C\Big(\sqrt{\frac{p}{n}} \vee \frac{p}{n}\Big)\, e^{-\tau^2/C}
+C\sqrt{\frac{s_0\vee \log p}{n}}
\, .
    \end{align}
\end{theorem}
\begin{proof}
The proof is obtained by adding the error terms from Propositions \ref{prop:noise},
\ref{prop:signal}, \ref{prop:cross} and \ref {prop:diag}, and noting that $\beta$ is bounded.
\end{proof}

Using Propositions \ref{prop:noise}, \ref{prop:signal}, \ref{prop:cross} and \ref{prop:diag}, together with a suitable choice of 
$\tau$, we obtain the proof of Theorem \ref{thm:est}.
\begin{proof}[Proof of Theorem \ref{thm:est}]
Note that in the case $s_0^2>p/e$ there is no thresholding and hence the result follows from the fact that $\|\hbSigma-\bSigma\|_{op}\le C\sqrt{p/n}$
\cite[Remark 5.40]{Vershynin-CS}. 

We assume now that $s_0^2\le p/e$ and the case that $\tau_* = C_1(\beta\vee 1)\sqrt{\log(p/s_0^2)} \le \sqrt{\log p}/2$. In that case we set $\tau = \tau_* \le \sqrt{\log p}/2$.
Below we will keep $C_1$ a large enough constant, and check that each of the error terms 
in Propositions \ref{prop:noise}, \ref{prop:signal}, \ref{prop:cross} and \ref{prop:diag} is upper bounded by (a constant times) the right-hand side of 
Eq.~(\ref{eq:EstTheorem}). Throughout $C$ will denote a generic constant that can be made as large as we want, and can change from line to line.

We start from Proposition \ref{prop:noise}:
\begin{align}
\norm{\bN}_{op} &\le C\,  \left(\sqrt{\frac{p}{n}} \vee \frac{p}{n}\right) \,\left(\frac{s_0^2}{p}\right)^C\\ 
&\le C\sqrt{\frac{p}{n}  \left(\frac{p}{s_0^2}\right)^{-C-1}} \;\vee\; C\sqrt{\left(\frac{p}{n}\right)^2  \left(\frac{p}{s_0^2}\right)^{-C-2}}\\
&\le C\sqrt{\frac{s_0^2}{n}  \left(\frac{p}{s_0^2}\right)^{-C}} \;\vee\; C\sqrt{\left(\frac{s_0^2}{n}\right)^2  \left(\frac{p}{s_0^2}\right)^{-C}}\\
&\le C\sqrt{\frac{s_0^2}{n}  \log \frac{p}{s_0^2}} \, , \label{eq:noiseBoundThmEst}
\end{align}
where in the last step we used $(e\, s_0^2/p), (s_0^2/n)\le 1$.

Next consider Proposition \ref{prop:signal}:
\begin{align}
\big\| \bS-\bSigma\big\|_{op}&\le C\sqrt{\frac{s_0^2\tau^2 }{n}}+C\sqrt{\frac{s_0(\beta\vee 1)^2}{n}}\\
&\le C\sqrt{\frac{s_0^2( \beta^2\vee 1)}{n}  \log \frac{p}{s_0^2}} \, .
\end{align}

From Proposition \ref{prop:cross}, we get, using the same argument as in \myeqref{eq:noiseBoundThmEst}
\begin{align}
    \norm{\b{C}}_{op} &\le C\sqrt{\beta\vee 1}\, \bigg(\sqrt{\frac{p}{n}}\vee \frac{p}{n}\bigg)  \, \left(\frac{s_0^2}{p}\right)^C\\
    &\le C(\beta\vee 1) \sqrt{\frac{s_0^2}{n}\log \frac{p}{s_0^2}}.
\end{align}

Finally, the term of Proposition \ref{prop:diag} is also bounded as desired using $\log p\le s_0^2\log (p/s_0^2)$
(dividing both sides by $p$ and using the fact that $x\mapsto x\log(1/x)$ is increasing).

The case of $\tau_* \ge \sqrt{\log p}/2$ is easier. In that case, we can keep $\tau=C_2\tau_*$
with $C_2$ large enough so that $\tau \ge C_0 (\beta\vee 1)\sqrt{\log p}$  for $C_0$ of Proposition
\ref{prop:largeThresh}. Then, by Proposition \ref{prop:largeThresh}, 
we know that $\b{N} = 0$ and $\b{C}= 0$. Therefore we only need consider the terms
$\b{S} - \bSigma$ and $\b{D}$. For these terms we can use Propositions \ref{prop:signal}
and \ref{prop:diag} respectively and, arguing 
as in the earlier case $\tau_* \le \sqrt{\log p}$, we obtain the desired
result. 
\end{proof}

\subsection{Proof of Proposition \ref{prop:noise}}\label{subsec:proofnoise}

Define $\btN$ as
\begin{align*}
  \btN &= \projnd\left\{\eta\left( \frac{1}{n}\bZ^\sT\bZ \right)\right\}.
\end{align*}
Since $\bN$ is a principal submatrix of $\btN$, it suffices to prove the same bound for
$\btN$. Our main tool in the proof will be the concentration lemma \ref{lem:nonconvexConc} which we
use on multiple occasions. With
a view to using the lemma, we let let $\bZ' \in \reals^{n\times p}$ denote an independent copy of $\bZ$, and $\btz'_i$ it's $i^\text{th}$
column. The proof relies on two preliminary lemmas. For some $A\ge 1$ (to be chosen later), we first state and
prove the following lemma that controls
the norm of \emph{any principal submatrix of }$\bt{N}$ of size at most $p/A$. 

\begin{lemma}
    \label{lem:smallSetLemma}
    Fix any $A\ge 1$. There exists an absolute constants $C, c$ such that:
    \begin{align}
        \P\Big\{ \max_{\sS \subseteq[p], |\sS|\le p/A} \|\proj_{\sS\times \sS} ( \bt{N}) \|_{op}\ge \Delta\Big\} &\le 
        C\, \exp\Big( p\frac{\log CA}{A} -  \frac{n^2\Delta^2}{C(n+p)} \Big) + C\frac{(np)^C}{\Delta^2} \exp(-cn).
    \end{align}
\end{lemma}
\begin{proof}
    For any subset $\sS \subset [p]$ recall that  $T^\eps_p(\sS)$ denotes an $\eps$-net of unit vectors in $\sphere^{p-1}$ supported
    on the subset $\sS$. For simplicity let $T(A) = \cup_{\sS:|\sS|\le p/A} T^\eps_p(\sS)$. 
    It suffices, by Lemma \ref{lem:specnormbnd}, to control $\<\b{y},\bt{N}\b{y}\>$ on the
    set $T(A)$. In particular:
    \begin{align}
        \P\Big\{ \max_{\sS \subseteq[p], |\sS|\le p/A} \|\proj_{\sS\times \sS} ( \bt{N})\|_{op} \ge \Delta\Big\} &\le  
        \P\Big\{ \max_{\b{y}\in T(A)} |\<\b{y}, \bt{N}\b{y}\>|  \ge \Delta(1-2\eps)\Big\}. \label{eq:epsNetSmall}
    \end{align}
    Consider the good set $\cG_1$ given by:
    \begin{align}
        \cG_1 &= \{ (\bZ, \bZ'): \max( \norm{\bZ}, \norm{\bZ'} ) \le  \sqrt{2}(\sqrt{n} + \sqrt{p}))\}.
        \label{eq:goodDefSmallSet}
    \end{align}
    To use Lemma \ref{lem:nonconvexConc}, we need to compute $\E\<\b{y}, \bt{N}\b{y}\>$ and the gradient of $\<\b{y}, \bt{N}\b{y}\>$
    with respect to the underlying random variables $\bZ$. Since $\eta(\cdot)$ is an odd function
    the expectation vanishes. To compute
    the gradient, we let $t\in[0, 1]$ and $\b{W} = \sqrt{t}\bZ + \sqrt{1-t}\bZ'$, 
    and consider $\<\b{y}, \bt{N}\b{y}\> = \<\b{y}, \eta(\b{W}^\sT\b{W}/n)\b{y}\>$ as a function of the
    $\b{W}$. Taking the gradient with respect to a column $\bt{w}_\ell$ for $\ell\in \sS$:
    \begin{align}
        \grad_{\bt{w}_{\ell}} \<\b{y}, \bt{N}\b{y}\> &= \frac{y_\ell}{n}\sum_{i\ne \ell, i \in \sS} \bt{w}_i y_i\partial\eta(\<\bt{w}_i, \bt{w}_\ell\>/ n) \\
        &= \frac{y_\ell}{n} \b{W} \bsigma,
    \end{align}
    where 
    \begin{align}
        \sigma_i &= \begin{cases}
            y_i \partial\eta(\<\bt{w}_i, \bt{w}_\ell\>/n) &\text{ if } i \ne \ell, i \in \sS \\
            0 &\text{otherwise.}
        \end{cases}
    \end{align}
    Since $\norm{\bsigma} \le \norm{\b{y}} = 1$, we have that $\norm{\grad_{\bt{w}_\ell}  \<\b{y}, \bt{N}\b{y}\>}^2 \le y_\ell^2 \norm{\b{W}}^2/n^2$.
    Summing over $\ell\in\sS$ we obtain the gradient bound, holding on the good set $\cG_1$:
    \begin{align}
        \norm{ \grad_{\b{W}}\<\b{y}, \bt{N}\b{y}\>}^2 &\le \frac{\sum_{\ell} y_{\ell}^2}{n^2} \norm{\b{W}}^2\\
        &\le \frac{C(n + p)}{n^2} \, ,
        \label{eq:gradientBoundSmallSet}
    \end{align}
    which holds because of triangle inequality and the fact that
    $\sqrt{t} + \sqrt{1-t} \le \sqrt{2}$. We can now apply Lemma \ref{lem:nonconvexConc}
    to bound the RHS of \myeqref{eq:epsNetSmall} and get:
    \begin{align}
        \P\Big\{ \max_{\sS \subseteq[p], |\sS|\le p/A} \proj_{\sS\times \sS} ( \bt{N}) \ge \Delta\Big\} &\le  
        C\abs{T(A)}\,\exp\Big( -\frac{n^2\Delta^2}{C(n+p)}\Big) \nonumber\\ &\quad+ \frac{C}{\Delta^2}\E\big\{\max_{\by\in T} \<\b{y}, \bt{N}\b{y}\>^2  ; \cG_1^c\big\}.
    \end{align}
    We can  simplify the terms on the right-hand side to obtain the result of the lemma. With $\eps = 1/4$, Stirling's
    approximation and Lemma \ref{lem:epsnetcard} we have:
    \begin{align}
        \abs{T(A)} &\le \exp\Big( p\frac{\log CA}{A}\Big).
    \end{align}
    We use a crude bound on the complement of the good set $\cG_1$. It is easy to see that, for any unit vector $\by$, 
    $\<\b{y}, \bt{N}\b{y}\>^2 \le \norm{\bt{N}}_F^2 \le \norm{\b{Z}^\sT \b{Z}}_F^2/n^2$. 
    Cauchy-Schwarz then implies that
    \begin{align}
        \E\{\max \<\b{y}, \bt{N}\b{y}\>^2  ; \cG_1^c\} &\le n^{-2}\big(\E\{\norm{\bZ^\sT \bZ}_F^4\}\big)^{1/2} \P\{\cG_1^c\}^{1/2} \\
        &\le (np)^C \exp(-c(n+p)),
    \end{align}
    where the bound on $\P\{\cG_1^c\}$ follows from Lemma \ref{lem:gaussianmatnorm}. This concludes the lemma.
\end{proof}

Note that Lemma \ref{lem:smallSetLemma}, with $A=1$, tells us that $\lVert{\bt{N}}\rVert_{op}$ 
is of order $\sqrt{p/n +(p/n)^2}$ (uniformly in $\tau$) with high probability. Already this non-asymptotic
bound is non-trivial, since
the previous results of \cite{cheng2013spectrum} and \cite{fan2015spectral} do 
not extend to this case. However, Proposition \ref{prop:noise} is stronger, 
and establishes a rate of decay with the thresholding level $\tau$. 

The second lemma we require controls the Rayleigh quotient $\<\by, \bt{N}\by\>$
when the entries of $\by$ are ``spread out''. 
\begin{lemma}
    \label{lem:smallEntriesNoise}
    Assume that $\tau \le \sqrt{\log p}/2$. Given $A\ge 1$ and a
    unit vector $\by$, let $\sS = \{i: \abs{y_i}\le \sqrt{A/p}\}$ and $\by_{\sS}, \by_{\sS^c}$ denote
    the projections of $\by$ onto supports $\sS, \sS^c$ respectively. We have:
    \begin{align}
        \P\Big\{ \max_{\b{y} \in T^{1/4}_{p}} \lvert \<\b{y}_{\sS}, \bt{N}\b{y}_{\sS}\> \rvert \ge \Delta\Big\}
        &\le C\exp\Big(- \frac{n^2\Delta^2}{L_1^2} + Cp\Big) + (np)^C\exp\big(-c\min(\sqrt{p}, n)\big), \label{eq:noiseBoundyss}
    \end{align}
    for any $\Delta\ge L_1$ where $L_1 = C_1 \sqrt{ A\exp(-\tau^2/16)(n+p)/n^2 }$.
   The same bound holds for $\P\big\{\max_{\by\in T^{1/4}_p}|\<\by_{\sS^c}, \bt{N}\by_{\sS}\> | \ge\Delta \big\}$.
\end{lemma}
\begin{proof}
    We first prove the claim for $\<\by_{\sS}, \btN\by_{\sS}\>$. 
    Firstly, we have $\E\<\b{y}_{\sS}, \bt{N}\by_{\sS}\> = 0$. Consider the ``good set'' $\cG_2$ of pairs $(\b{W}, \b{W'}) \in \reals ^{n\times p} \times \reals^{n\times p}$ 
    satisfying the conditions:
    \begin{align}
        \norm{\b{W}}, \norm{\b{W}'} \le \sqrt{2}(\sqrt{n} + \sqrt{p}) \label{eq:normConditionG2}\, ,\\
        \forall i \in [p], \quad \frac{1}{p} \sum_{j\in [p]\bsl i} |\ind(\<\bt{w}_i, \bt{w}_j\>| \ge \tau\sqrt{n}/2) &\le 2\exp(-\tau^2/16)\label{eq:rowl0normG2}\, ,\\ %
        \forall i \in [p], \quad \frac{1}{p} \sum_{j\in [p]\bsl i} |\ind(\<\bt{w}'_i, \bt{w}'_j\>| \ge \tau\sqrt{n}/2) &\le 2\exp(-\tau^2/16) \label{eq:rowprimel0normG2}\, ,\\ %
        \forall i \in [p], \quad \frac{1}{p}\sum_{j\in [p]} \ind( |\<\bt{w}_i, \bt{w}'_j\>| \ge \tau\sqrt{n}/2) &\le 2\exp(-\tau^2/16). \label{eq:rowcrossl0normG2}\, .%
    \end{align}
    Also, for any pair $\b{W}, \b{W'} \in \cG_2$, for $\b{W}(t) = \sqrt{t}\b{W} + \sqrt{1-t}\b{W'}$ (and its columns
    $\bt{w}(t)_i$ defined appropriately) we have:
    \begin{align}
        \norm{\b{W}(t)} \le \max_{t}(\sqrt{t} + \sqrt{1-t})(\sqrt{2n} + \sqrt{2p}) = 2(\sqrt{n} + \sqrt{p}),   \label{eq:conditionOnPath1}\\
        \forall i \in [p] \quad \frac{1}{p} \sum_{j \in [p]\bsl i} \ind( \<\bt{w}(t)_i, \bt{w}(t)_{j}\> \ge \tau\sqrt{n} ) \le 6\exp(-\tau^2/16).  \label{eq:conditionOnPath2}
    \end{align}
    Equation (\ref{eq:conditionOnPath1}) follows by a simple application of triangle inequality and condition \eqref{eq:normConditionG2} defining
    $\cG_2$. For inequality \eqref{eq:conditionOnPath2}, expanding the product $\<\bt{w}(t)_i, \bt{w}(t)_j\>$:
    \begin{align}
        \<\bt{w}(t)_{i}, \bt{w}(t)_{j}\> &= t \<\bt{w}_{i}, \bt{w}_j\> + (1-t)\<\bt{w}'_i, \bt{w}'_j\> + \sqrt{t(1-t)}\<\bt{w_i}, \bt{w}'_j\>,
    \end{align}
    whence, by triangle inequality and $\sqrt{t(1-t)}< 1$
    \begin{align}
        \ind(\abs{\<\bt{w}(t)_i, \bt{w}(t)_j\>} \ge \tau\sqrt{n}) &\le \ind( \lvert \<\bt{w}_i, \bt{w}_j\>\rvert \ge \tau\sqrt{n}/2) +
        \ind (\lvert \<\bt{w}'_i, \bt{w}'_j\> \rvert \ge \tau\sqrt{n}/2)\nonumber \\ %
        &\quad + \ind(\lvert\<\bt{w}_i, \bt{w}'_j\>\rvert \ge \tau\sqrt{n}/2).
        \label{eq:sparsityOnPath}
    \end{align}
    The gradient of $\<\by_{\sS}, \eta(\b{W}^\sT\b{W}/n)\by_{\sS}\>$ with respect to a column $\bt{w}_{\ell}$ of $\b{W}$ is given by:
    \begin{align}
        \grad_{\bt{w}_{\ell}} \<\by_{\sS}, \eta(\b{W}^\sT\b{W}/n) \b{y}_{\sS}\> &= \frac{y_{\ell}}{n} \sum_{j\in \sS\bsl \ell} y_j\partial\eta\big( %
        \frac{\<\bt{w}_j, \bt{w}_\ell\>}{n} ; \frac{\tau}{\sqrt{n}} \big) \bt{w}_{j} \\
        &= \frac{y_{\ell}}{n} \b{W} \bsigma, \\
        \text{ where }\; \sigma_{i} &= \begin{cases}
            \partial \eta(\<\bt{w}_i, \bt{w}_{\ell}\>/n; \tau/\sqrt{n}) y_{i} &\text{ when } i \in \sS\bsl \ell\\
            0 &\text{ otherwise.}
        \end{cases}
    \end{align}
    Therefore
    \begin{align}
        \norm{ \grad_{\bt{w}_{\ell}} \<\by_{\sS}, \btN\by_{\sS}\>}^2 &\le \frac{y_{\ell}^2}{n^2} \norm{\b{W}}^2 \norm{\bsigma}^2 \\
        &\le \frac{y_{\ell}^2 \norm{\b{W}}^2}{n^2} \sum_{i\ne \ell} (y_i \partial\eta(\<\bt{w}_i, \bt{w}_\ell\>/n))^2 \\
        &\stackrel{(a)}{\le} \frac{y_\ell^2 \norm{\b{W}}^2}{n^2} \sum_{i\ne \ell} \frac{A}{p} \ind( \abs{\<\bt{w}_{i}, \bt{w}_\ell\>}\ge \tau\sqrt{n}) \\
        &\stackrel{(b)}{\le} \frac{y_\ell^2 }{n^2} C(n+p) A\exp(-\tau^2/16)
    \end{align}
    Here $(a)$ follows from fact that the entries of $\by_{\sS}$ are bounded by $\sqrt{A/p}$ and the 
    definition of the soft thresholding function. Inequality $(b)$ follows
    follows when we set $\b{W} = \b{Z}(t) = \sqrt{t}\b{Z} + \sqrt{1-t}\b{Z}'$ and $(\bZ, \bZ') \in \cG_2$. 
    Therefore, summing over $\ell$ we obtain the following bound for the gradient of $\<\by_{\sS}, \btN\by_{\sS}\>$
    \begin{align}
        \norm{ \grad_{\b{Z}(t)}\<\by_{\sS}, \btN\by_{\sS}\>}^2 &\le C_1\frac{A\exp(-\tau^2/16)(n+p)}{n^2} \equiv L_1^2.
    \end{align}
    We can use now Lemma \ref{lem:nonconvexConc}, to get, for $L_1>0$ as defined above and any $\Delta\ge L_1$:
    \begin{align}
        \P\Big\{ \max_{\by\in T^{1/4}_p}\<\by_{\sS}, \btN\by_{\sS}\> \ge \Delta\Big\} &\le C\exp\Big(- \frac{\Delta^2}{CL_1^2} + Cp\Big)\nonumber\\%
        &\quad + CL_1^{-2} \E\{\max_{\by\in T^\eps_p} \<\by_{\sS}, \btN\by_{\sS}\>^2 ; \cG_2\} \\
        &\le C\exp\Big(- \frac{\Delta^2}{CL_1^2} + Cp\Big)
        + C(np)^C\P\{\cG_2^c\}^{1/2}, \label{eq:finalBound}
    \end{align}
    where the last line follows by Cauchy-Schwarz, as in the proof of Lemma \ref{lem:smallSetLemma},
    and the fact that $L_1 \ge (np)^{-C_2}$ using the upper bound $\tau \le \sqrt{\log p}/2$.

    To obtain the thesis, we need to now bound $\P\{\cG_2^c\}$. It suffices
    to control the failure probability of conditions \eqref{eq:normConditionG2}, \eqref{eq:rowl0normG2}, \eqref{eq:rowprimel0normG2}, \eqref{eq:rowcrossl0normG2} 
    of the good set $\cG_2$
    individually, and apply the union bound. For $\b{Z}, \b{Z'}$ independent, $\max(\norm{\bZ}, \norm{\bZ'}) \ge \sqrt{2}(\sqrt{n}+\sqrt p)$
    with probability at most $2\exp(-c(n+p))$ by Lemma \ref{lem:gaussianmatnorm}. 
    Now consider condition \eqref{eq:rowl0normG2} with $i = 1$, without loss of generality. First, for any $h>0$ we have:
    \begin{align}
        \P\Big\{ \frac{1}{p} \sum_{j \ne 1} \ind(\abs{\<\btz_1, \btz_j\> } \ge \tau\sqrt{n}/2) \ge  h \Big\} 
        & \le \P \Big\{ \frac{1}{p} \sum_{j \ne 1} \ind(\abs{\<\btz_1, \btz_j\> } \ge \tau\sqrt{n}/2) \ge  2h ; \norm{\btz_1} \le 2\sqrt{n}\Big\} \nonumber \\%
        &\quad+ \P \big\{ \norm{\btz_1} \ge \sqrt{2n}\big\}. \label{eq:smallLargeNorm}
    \end{align}
    Lemma \ref{lem:chisquare} guarantees that the second term is at most $\exp(-cn)$. 
    To control the first term, we note that, conditional on $\btz_1$, $\<\btz_j, \btz_1\>, j\ne 1$ are independent Gaussian
    random variables with variance $\norm{\btz_1}^2$. Therefore, conditional on $\btz_1$,  $\ind(|\<\btz_1, \btz_j\>|\ge \tau\sqrt{n}/2)$
    are independent Bernoulli random variables with success probability $h_0 = 2\Phi\big(-\tau\sqrt{n}/(2\norm{\bt{z}_1})\big)$, where $\Phi(\cdot)$ is the
    Gaussian cumulative distribution function. 
    It follows, by the Chernoff-Hoeffding bound for Bernoulli random variables that 
    \begin{align}
        \P \Big\{ \frac{1}{p} \sum_{j \ne 1} \ind(\abs{\<\btz_1, \btz_j\> } \ge \tau\sqrt{n}/2) \ge  h \big\vert \btz_1 \Big\}
&\le \exp\big( - p \, D( h  \Vert h_0) \big),
    \end{align}
    where $D(a \Vert b) = a\log (a/b) + (1-a)\log[(1-a)/(1-b)]$. Choosing $h = 4\Phi(-\tau/(2\sqrt{2}))$, and
    conditional on $\norm{\bt{z}_1}\le \sqrt{2n}$, $D(h  \Vert h_0) \ge c h $ for a constant $c$, implying
    that
    \begin{align}
        \P \Big\{ \frac{1}{p} \sum_{j \ne 1} \ind(\abs{\<\btz_1, \btz_j\> } \ge \tau\sqrt{n}/2) \ge  h ; \norm{\btz_1}\le \sqrt{2n} \Big\}
        &\le \exp(-cph).
    \end{align}
    By standard bounds $h = 4\Phi(-\tau/2\sqrt{2}) \le 2\exp(-\tau^2/16)$ and, as $\tau\le \sqrt{\log p}/2$, $h \ge 1/\sqrt{p}$, we have
    \begin{align}
        \P \Big\{ \frac{1}{p} \sum_{j \ne 1} \ind(\abs{\<\btz_1, \btz_j\> } \ge \tau\sqrt{n}/2) \ge  h ; \norm{\btz_1}\le \sqrt{2n} \Big\}
        &\le \exp(-c\sqrt{p}).
    \end{align}
    Combining this with \myeqref{eq:smallLargeNorm} we now get:
    \begin{align}
        \P\Big\{ \frac{1}{p} \sum_{j \ne 1} \ind(\abs{\<\btz_1, \btz_j\> } \ge \tau\sqrt{n}/2) \ge  h \Big\} 
        &\le 2\exp(-c\min(n, \sqrt{p})).
    \end{align}
    A similar bound holds for $i\ne 1$ and the other conditions \eqref{eq:rowprimel0normG2} and \eqref{eq:rowcrossl0normG2},
    whence we have by the union bound that $\P\{\cG_2^c\}\le p^2\exp(-c\min(\sqrt{p}, n))$. This completes
    the proof of the claim \eqref{eq:noiseBoundyss}.

    The proof of the claim for $\<\by_{\sS}, \bt{N}\by_{\sS^c}\>$ is analogous, so we only sketch the points at which it differs from that
    of \myeqref{eq:noiseBoundyss}. We use the same good set $\cG_2$, as defined earlier. Computing the gradient as for $\<\by_{\sS}, \btN\by_{\sS}\>$
    we obtain:
    \begin{align}
        \grad_{\bt{w}_{\ell}}\<\by_{\sS}, \btN\by_{\sS^c}\> &= \frac{y_\ell}{n}\sum_{j\in \sS(\ell)} y_{j} \bt{w}_j \partial \eta\Big(\frac{\<\bt{w}_j, \bt{w}_\ell\>}{n}; \frac{\tau}{\sqrt{n}}\Big).
    \end{align}
    Here $\sS(\ell) = \sS^c$ if $\ell \in\sS$ and $\sS$ otherwise. Define the vector $\bsigma(\ell) \in \reals^{p}$ as
    \begin{align}
        (\bsigma(\ell))_j &= \begin{cases}
            y_\ell y_j\partial\eta\Big( \frac{\<\bt{w}_j, \bt{w}_\ell\>}{n} ; \frac{\tau}{\sqrt{n}} \Big) &\text{ if } j \in \sS(\ell) \\
            0 &\text{ otherwise.}
        \end{cases}
    \end{align}
    As before, we have that $\norm{\grad_{ \bt{w}_{\ell}}\<\by_{\sS}, \btN\by_{\sS^c}\>  } = n^{-1}\norm{\b{W}\bsigma(\ell) } \le n^{-1}\norm{\b{W}}\norm{\bsigma(\ell)}$.
    Therefore, summing over $\ell\in [p]$:
    \begin{align}
        \norm{\grad_{\b{W} }\<\by_{\sS}, \btN\by_{\sS^c}\>  }^2 &\le \frac{\norm{\b{W}}}{n^2}\sum_{\ell\in [p]} \norm{\bsigma(\ell)}^2 \\
        &\le \frac{\norm{\b{W}}^2}{n^2}\sum_{\ell\in [p]} \sum_{j\in \sS(\ell)} y_\ell^2y_j^2 \partial\eta \Big( \frac{\<\bt{w}_j, \bt{w}_\ell\>}{n} ; \frac{\tau}{\sqrt{n}} \Big)\\
        &= \frac{2\norm{\b{W}}^2}{n^2} \sum_{\ell\in \sS} \sum_{j\in \sS^c} y_j^2y_\ell^2 \partial\eta \Big( \frac{\<\bt{w}_j, \bt{w}_\ell\>}{n} ; \frac{\tau}{\sqrt{n}} \Big) \\
        &\le \frac{2\norm{\b{W}}^2}{n^2}\frac{A}{p}\max_{\ell\in[p]} \sum_{j\ne p} \partial\eta \Big( \frac{\<\bt{w}_j, \bt{w}_\ell\>}{n} ; \frac{\tau}{\sqrt{n}} \Big).
    \end{align}
    Under the condition of $\cG_2$, the gradient also satisfies, when evaluated at $\b{W} = \bZ(t)= \sqrt{t}\bZ + \sqrt{1-t}\bZ'$: 
    \begin{align}
        \norm{\grad_{\b{Z}(t) }\<\by_{\sS}, \btN\by_{\sS^c}\>}^2 &\le \frac{CA\exp(-\tau^2/16) (n+p)}{n^2}.
    \end{align}
    The rest of the proof is then the same as before. 
\end{proof}

Given these lemmas, we can now establish Proposition \ref{prop:noise}.
\begin{proof}[Proof of Proposition \ref{prop:noise}]
We use a variant of the $\eps$-net argument of Lemma \ref{lem:smallSetLemma}.
To bound the probability that $\lVert{\bt{N}}\rVert_{op}$ is large, with Lemma
\ref{lem:specnormbnd}, we obtain:
\begin{align}
    \P\big\{ \lVert{\bt{N}\rVert_{op}} \ge \Delta\big\}   &\le \P\Big\{ \max_{\by\in T^{\eps}_{p}} \abs{\<\b{y}, \bt{N}\b{y}\> }
\ge \Delta(1-2\eps)\Big\}.
\end{align}
Let $\sS = \{i: \abs{y_i} \le \sqrt{A/p}\}$ for some $A\ge 1$ to be
chosen later. Then let $\by = \by_{\sS} + \by_{\sS^c} $ denote
the projections of $\by$ onto supports $\sS, \sS^c$ respectively. Since
$\<\b{y}, \bt{N}\b{y}\> = \<\b{y}_{\sS^c}, \bt{N}\b{y}_{\sS^c}\> + \<\b{y}_{\sS}, \bt{N}\b{y}_{\sS}\> + 2\<\b{y}_\sS, \bt{N}\b{y}_{\sS^c}\>$ 
by triangle inequality and union bound:
\begin{align}
    \P\big\{ \lVert{\bt{N}\rVert_{op}} \ge \Delta\big\} &\le 
    \P\Big\{ \max_{\by \in T^{\eps}_{p}} \lvert \<\b{y}_{\sS^c}, \bt{N}\b{y}_{\sS^c}\> \rvert 
    + \lvert \< \by_{\sS}, \bt{N}\b{y}_{\sS}\>\rvert  +2 \lvert \<\b{y}_{\sS}, \bt{N}\b{y}_{sS^c}\> \rvert  \ge \Delta(1-2\eps) \Big\} \\
    &\le \P\Big\{ \max_{\by \in T^{\eps}_{p}} \lvert \<\b{y}_{\sS^c}, \bt{N}\b{y}_{\sS^c}\> \rvert \ge \Delta(1-2\eps)/4\Big\} 
    + \P\Big\{ \max_{\by \in T^{\eps}_{p}} \vert\<\b{y}_{\sS}, \bt{N}\b{y}_{\sS}\> \rvert  \ge \Delta(1-2\eps)/4\Big\} \nonumber\\
    &\quad + \P\Big\{ \max_{\by \in T^{\eps}_{p}} \vert\<\b{y}_{\sS}, \bt{N}\b{y}_{\sS^c}\> \rvert  \ge \Delta(1-2\eps)/4\Big\} \\
    &\le \P\Big\{ \max_{\sS': \abs{\sS'} \le p/A} \norm{ \proj_{\sS'\times\sS'} (\bt{N})} \ge \Delta(1-2\eps)/4\Big\} 
   + \P\Big\{ \max_{\by \in T^{\eps}_{p}} \lvert\<\b{y}_{\sS}, \bt{N}\by_{\sS}\>    \rvert  \ge \Delta(1-2\eps)/4\Big\} \nonumber \\
   &\quad + \P\Big\{\max_{\by\in T^{\eps}_p}\lvert \<\b{y}_{\sS}, \bt{N}\b{y}_{\sS^c}\> \rvert \ge \Delta(1-2\eps)/4   \Big\}\label{eq:smallLargeDecomp}.
\end{align}
With $\eps=1/4$, the first term is controlled by 
Lemma \ref{lem:smallSetLemma}  while the 
final two are controlled by Lemma \ref{lem:smallEntriesNoise}. 
We choose $\eps=1/4$ in \myeqref{eq:smallLargeDecomp}, and 
\begin{align}
    \Delta = \Delta_* &\equiv C\sqrt{\frac{p}{n}\left( 1+\frac{p}{n} \right)\left( \frac{\log A}{A} + A\exp\Big(-\frac{\tau^2}{16}\Big) \right)},
\end{align}
for large enough $C$ so that, using the bounds of Lemmas \ref{lem:smallSetLemma} and \ref{lem:smallEntriesNoise}, we have:
\begin{align}
\P\big\{\bt{N} \ge \Delta_*\big\} &\le C(np)^C \exp\Big[ - c\min\Big(\sqrt{p}, n, p\frac{\log A}{A}\Big)  \Big].
\end{align}
This probability bound is $o(1)$ provided $A$ is not too large: we choose $A = 0.25\sqrt{\tau\exp(\tau^2/16)} \ll \sqrt{p}$
which guarantees that the bound above is $o(1)$ when $n > C\log p$ for some $C$ large enough. This concludes the proposition.
\end{proof}

\subsection{Proof of Proposition \ref{prop:signal}}\label{subsec:proofsignal}

We decompose the empirical covariance matrix (\ref{eq:empcov}) as
\begin{align}
 \proj_\sE(\hbSigma) &= \bSigma + \bDelta_1+\bDelta_2+\bDelta_2^{\sT} +\proj_\sE\Big(\frac{1}{n}\bZ^\sT\bZ-
                       \id_p\Big)\, ,
\label{eq:decompSignal}\\
\bDelta_1 & \equiv \sum_{q,q'=1}^r \sqrt{\beta_q\beta_q'} \Big(\frac{1}{n}\<\bu_q,\bu_q'\>-\bfone_{q=q'}\Big)\bv_q\bv_q'^{\sT}\, ,\\
\bDelta_2 & \equiv \sum_{q=1}^r\frac{\sqrt{\beta_q}}{n}\bv_q(\bZ^{\sT}\bu_q)^{\sT}_{\sQ} \, .
\end{align}
Next notice that, for any $x\in\reals$, 
\begin{align}
    \big|\eta(x)-x\big|\le \frac{\tau}{\sqrt{n}}\, .\label{eq:etaBias}
\end{align}
With a view to employing this inequality, we use
\myeqref{eq:decompSignal} and the triangle inequality:
\begin{align}
\big\| \proj_\sE(\eta(\hbSigma))-\bSigma\big\|_{op} &=
\Big\| \proj_{\sE}\big(\eta(\hbSigma)\big) - \proj_{\sE}\big(\hbSigma\big) - \bDelta_1 -\bDelta_2-\bDelta_2^\sT - \proj_{\sE}\Big(\frac{1}{n}\bZ^\sT\bZ - \id_p \Big)\Big\|_{op} \\
&\le \big\|\proj_{\sE} \big(\eta(\hbSigma) - \hbSigma\big) \big\|_{op}
+ \norm{\bDelta_1}_{op} + 2\norm{\bDelta_2}_{op} + \Big\|\proj_{\sE}\Big(\frac{1}{n}\bZ^\sT\bZ-\id_p\Big)\Big\|_{op} \\
&\le \frac{s_0\tau}{\sqrt{n}} + \norm{\bDelta_1}_{op} + 2\norm{\bDelta_2}_{op} + \Big\|\proj_{\sE}\Big(\frac{1}{n}\bZ^\sT\bZ-\id_p\Big)\Big\|_{op}, \label{eq:SignalDecompose}
\end{align}
where the last line follows by noticing that the first term is supported on
$\sE$ of size $s_0\times s_0$ and then using bias bound \myeqref{eq:etaBias} entry-wise.
%
%
We next bound each of the three terns on the right hand side. 
%

For the first term in Eq~(\ref{eq:SignalDecompose}), note that with a 
change of basis to the orthonormal set $\bv_1,\dots\bv_r$  $\bDelta_1$
is equivalent to an $r\times r$ matrix with entries $M_{qq'}\sqrt{\beta_q\beta_{q'}}$, where $M_{qq'} = \big(\<\bu_q,\bu_q'\>/n-\bfone_{q=q'}\big)$. Denote by $\b{B}\in \reals^{r\times r}$ the diagonal
matrix with $B_{qq} = \sqrt{\beta_q}$ and
by $\b{U}\in\reals^{r\times n}$, the matrix with columns $\bu_1$,\dots $\bu_r$. 
Then, we have, with high probability
\begin{align}
    \|\bDelta_1\|_{op} & = \|\b{B}\b{M} \b{B}\|_{op} \\
    &\le \norm{\b{B}}_{op}^2 \norm{\b{M}}_{op} 
= \beta\|\frac{1}{n}\b{U}^{\sT}\b{U}-\id_{r\times r}\big\|_{op}\\
& \le C\beta\sqrt{\frac{r}{n}}\, .
\end{align}
The last inequality follows from the Bai-Yin law on eigenvalues of Wishart matrices \cite[Corollary 5.35]{Vershynin-CS}.

Consider the second term in Eq~(\ref{eq:SignalDecompose}).  By orthonormality of $\bv_1,\dots,\bv_r$, the matrix $\bDelta_2$ is orthogonally 
equivalent to $\b{B}\bZ_{\sQ}^{\sT}\bU/n$, where we recall that $\bZ_{\sQ}$ denotes the submatrix of $\bZ$
formed by the columns in $\sQ$. Denoting by  $\bP_{U}$ the orthogonal projector onto
the column space of $\bU$, we then have, with high probability,
\begin{align}
\|\bDelta_2\|_{op}&\le \frac{1}{n}\, \|\b{B}\|_{op}\|\bZ_{\sQ}^{\sT}\bP_{U} \bU\|_{op}\\
&\le \frac{\beta}{n}\,\|\bP_{U}\bZ_{\sQ}\|_{op}\| \bU\|_{op}\\
&\le \frac{C\beta}{n}\, \big(\sqrt{s_0}+\sqrt{r}\big) \big(\sqrt{n}+\sqrt{r}\big) \le C\beta \sqrt{\frac{s_0}{n}}\, .
\end{align}
Here the penultimate inequality follows by Lemma \ref{lemma:NormWishart} noting
that, by invariance under rotations (and since $\bP_U$ project onto a random subspace of $r$ dimensions
independent of $\bZ$), $\|\bP_{U}\bZ_{\sQ}\|_{op}$ is distributed as the norm of a matrix with i.i.d. standard normal
entries, with dimensions $|\sQ|\times r$, $|\sQ|\le s_0$.

Finally, for the third term of \myeqref{eq:SignalDecompose} we use the Bai-Yin law of
Wishart matrices \cite[Corollary 5.35]{Vershynin-CS} to obtain, with high probability:
\begin{align}
\Big\|\proj_\sE\Big(\frac{1}{n}\bZ^\sT\bZ - \id_p\Big)\Big\|_{op}&=
\Big\|\frac{1}{n}\bZ_{\sC}^{\sT}\bZ_{\sC} - \id_{s_0} \Big\|_{op} \\
&\le C\sqrt\frac{s_0}{n},
\end{align}

Finally, substituting the above bounds in Eq.~(\ref{eq:SignalDecompose}), we get
\begin{align}
\big\| \proj_\sE(\eta(\hbSigma))-\bSigma\big\|_{op} &=
\frac{\tau s_0}{\sqrt{n}} + C(1+\beta) \sqrt{\frac{s_0}{n}}, 
\end{align}
which implies the proposition.

\subsection{Proof of Proposition \ref{prop:cross}} \label{subsec:proofcross}

Note that $\b{C} = \btC+\btC^{\sT}$ where $\btC = \proj_{\sQ\times \sQ^c} \big(\eta(\hbSigma) \big)$. It is therefore sufficient to 
control $\btC$, and then use triangle inequality.
The proof is similar to that of Proposition \ref{prop:noise}. We let $\bU\in\reals^{n\times r}$ denote the matrix with columns $\bu_1$,
$\bu_2$,\dots $\bu_r$, and introduce the set 
\begin{align}
\cU \equiv \Big\{\bU\in\reals^{n\times r} :\;\; \Big\|\frac{1}{n}\bU^{\sT}\bU -\id_{r\times r}\Big\|_{op}\le 5\sqrt{\frac{r}{n}} \Big\}\, .
\end{align}
We then have
\begin{align}
\prob\big(\|\btC\|_{op}\ge \Delta \big) \le \sup_{\bU\in\cU }\prob\big(\|\btC\|_{op}\ge \Delta\, \big|\, \bU \big)  + \prob\big(\bU\not\in\cU\big)\, .
\end{align}
Notice that, by the Bai-Yin law on eigenvalues of Wishart matrices  \cite[Corollary 5.35]{Vershynin-CS}, 
$\lim_{n\to\infty}\prob(\bU\in \cU)=1$ (throughout $r<c\,n$ for $c$ a small constant). 
It is therefore sufficient to show  $\sup_{\bU\in\cU }\prob\big(\|\btC\|_{op}\ge \Delta\, \big|\, \bU \big)  \to 0$ for $\Delta$ as in the statement of
the theorem. 

In order to lighten the notation, we will write $\tprob(\,\cdot\,) \equiv \prob(\,\cdot\,|\bU)$ and bound the above probability uniformly over 
$\bU\in\cU$. (In other words $\tprob$ denotes expectation over $\bZ$ with $\bU$ fixed).
We first control the norms of small submatrices of $\btC$, following
which we control the full matrix. 
\begin{lemma}\label{lem:smallSetCross}%
    Fix an $A \in [1, p^{1/3}]$,  and let $L =\sqrt{((\beta\vee 1)n+p)/n^2}$. Then, there exists an absolute constant $C>0$ such that, for any $\Delta>0$:
\begin{align}
        \tprob\Big\{\max_{\sC^c \supseteq \sS: |\sS|\le p/A} \|\proj_{\sC\times\sS} \big(\eta(\hbSigma)\big)\|_{op} \ge \Delta \Big\}
        &\le C\exp\Big(C s_0 +  \frac{p\log (CA)}{A} - \frac{\Delta^2}{C L^2} \Big) %
        \nonumber\\%
        &\quad + L^{-2}(np)^C\exp(-n/C).
\end{align}
\end{lemma}
\begin{proof}
    Let, as before, $T^\eps_p(\sS)$ denote the $\eps$-net of unit vectors supported on $\sS \subset \sC^c$ of size at most $p/A$ and let
    $T = \cup_{\sS} T^\eps_p (\sS)$. 
    Then, by
    Lemma \ref{lem:specnormbnd}, with $\eps=1/4$:
    \begin{align}
        \tprob\Big\{\max_{\sS\subseteq\sC^c |\sS|\le p/A} \big\|\proj_{\sC\times\sS} \big(\eta(\hbSigma)\big) \big\|_{op} \ge \Delta \Big\}
        &\le 
        \tprob\Big\{\max_{\b{y}\in T, \b{w}\in T^\eps_{s_0}} \<\b{w},\btC\b{y}\>   \ge \Delta (1-2\eps)/2 \Big\}.
    \end{align}
    It now suffices to control the right hand side via Lemma \ref{lem:nonconvexConc}. We first compute
    the gradients with respect to $\bt{z}_\ell$ as before:
    \begin{align}
        \grad_{\bt{z}_\ell} \<\b{w}, \btC\b{y}\> &= \begin{cases}
            \frac{w_\ell}{n} \sum_{i\in \sC^c} y_i \partial\eta(\<\bt{x}_\ell, \bt{z}_i\>/n) \bt{z}_j &\text{ when } \ell\in \sC, .\\
            \frac{y_\ell}{n} \sum_{i\in \sC} w_i\partial\eta(\<\bt{z}_\ell, \bt{x}_i\>/n)\bt{x}_i &\text{ when } \ell\in \sC^c,  
        \end{cases} 
    \end{align}
    Therefore, arguing as in proof of Proposition \ref{prop:noise} (see Lemma \ref{lem:smallSetLemma}):
    \begin{align}
   \norm{\grad_{\bZ} \<\b{w}, \btC\b{y}\>}^2_F =\sum_{\ell}\norm{\grad_{\bt{z}_\ell} \<\b{w}, \btC\b{y}\>}^2 &\le \frac{\norm{\bZ}^2 + \norm{\b{X}_\sC}^2}{n^2}\,.
\label{eq:GradBound}
    \end{align}
Let $\b{B}\in\reals^{r\times r}$ be the diagonal matrix with entries $B_{q,q} = \sqrt{\beta_q}$, and $\b{V}\in\reals^{p\times r}$ be the matrix with columns
$\bv_1,\dots,\bv_r$. We then have $\bX = \bU\b{B}\b{V}^{\sT}+\bZ$, whence, recalling $\bU\in \cU$, and $r\le c\, n$ with $c$ small enough
\begin{align}
    \|\bX_\sC\|&\le \norm{\bX} \le  \|\bU\b{B}\b{V}^{\sT} \| +\|\bZ\|\\
&\le \sqrt{\beta}\|\bU\|  +\|\bZ \|\le 5\sqrt{\beta n} + \|\bZ\|\, . \label{eq:BoundBx}
\end{align}
Consider the good set $\cG_4$ of pairs $\big( \bZ, \bZ'\big)$    satisfying:
\begin{align}
        \max(\norm{\bZ}, \norm{\bZ'}) &\le \sqrt{2n} + \sqrt{2p} \, ,\\
        \max(\norm{\bZ_\sC}, \norm{\bZ'_\sC}) &\le \sqrt{2n} + \sqrt{2k}\, .
\end{align}
For $  \big( ( \bZ,  \bZ'\big)\in \cG_4$,     and $t\in [0, 1]$, define $\bZ(t) = \sqrt{t}\bZ + \sqrt{1-t}\bZ'$. Now 
Using Eqs.~(\ref{eq:GradBound}) and (\ref{eq:BoundBx}, the gradient $\grad \<\b{w}, \btC\b{y}\>$ evaluated at $\bZ(t)$
    satisfies:
    \begin{align}
        \norm{\grad \<\b{w}, \btC\b{y}\>}^2 &\le \frac{3\norm{\bZ(t)}^2 + 10\beta n}{n^2} \\
        &\le C\frac{(n+p) +\beta n }{n^2} \\
        &\le C\frac{(\beta\vee 1)n+p}{n^2}.
    \end{align}
 Now applying Corollary \ref{corr:nonconvexCS}, 
    for $L = C \sqrt{((\beta\vee 1)n+p)/n^2}$:
    \begin{align}
        \tprob\Big\{\max_{\sS\subseteq\sC^c |\sS|\le p/A} \big\lVert \proj_{\sC\times\sS} \big(\eta(\hbSigma)\big) \big\rVert_{op}\ge \Delta \Big\}
        &\le C|T|\, \exp\Big( - \frac{\Delta^2}{C L^2} \Big) \nonumber\\%
       &\quad + CL^{-2}\tE\{\max_{\b{w}, \by}\<\b{w}, \btC\by\>^4\}^{1/4} \P\{\cG_4\}^{1/2}. \label{eq:smallSetCrossBound1}
    \end{align}
    Let $\eps=1/4$, observing that $T \subseteq \cup_{\sS : |\sS| \le p/A} T^\eps_p(\sS)$, we have the bound (using Lemma \ref{lem:epsnetcard} and Stirling's approximation):
    \begin{align}
        |T| &\le \exp(Cs_0 +  A^{-1}p\log CA),   
    \end{align}
    for some absolute $C$. Now, as in the proof of Proposition \ref{prop:noise}, $|\<\b{w}, \btC\b{y}\>| \le \norm{\b{C}}\le \norm{\b{C}}_F\le \norm{\hbSigma}_F$. 
    From this it follows that $\tE\big\{\max_{\b{w}, \b{y}} \<\b{w}, \btC\b{y}\>^4\big\} \le (np)^{C}$ for some $C$. Finally
    $\P\{\cG_4^c\} \le \exp(-cn)$ using Lemmas \ref{lem:gaussianmatnorm}, \ref{lem:chisquare} and the union bound. 
    Combining these bounds in \myeqref{eq:smallSetCrossBound1} yields the lemma. 
\end{proof}

Now we prove a similar lemma when $\by$ has entries that are ``spread out''.
\begin{lemma}%
    \label{lem:spreadOutCross}
    Fix an $A \in [1, p^{1/3}]$, and a unit vector $\by\in \reals^{\sC^c}$ let $\sS = \{i: \abs{y_i}\le \sqrt{A/p}$,
    and $\by_{\sS}$ denote the projection of $\by$ on the set of indices $\sS$. Then there exists a numerical constant $C$ such that,
    assuming $\tau\le \sqrt{\log p}/2$, we have
 \begin{align}
        \tprob\Big\{ \max_{\b{w}\in T^{\eps}_{\sC}, \by\in T^\eps_{\sC^c}} \<\b{w}, \btC\by_{\sS}\> &\ge \Delta \Big\}   
        &\le C\exp\Big(-\frac{\Delta^2}{CL_*^2} + Cp\Big) + (np)^C\exp\big(-c\min(\sqrt{p}, n)\big),
    \end{align}
    where $L_* = \sqrt{A \exp(-\tau^2/C(\beta\vee 1)) (n(\beta\vee 1)+p)/n^2}$.
\end{lemma}
\begin{proof}
    For simplicity of notation, it is convenient to introduce the vector $\by' = \by_{\sS}$. Throughout the proof, we will use that $\norm{\by'}\le 1$ and
    $\norm{\by'}_{\infty} \le \sqrt{A/p}$. 
    We compute the gradients as follows:
    \begin{align}
        \grad_{\bt{z}_\ell} \<\b{w}, \btC\b{y}'\> &= \begin{cases}
            \frac{w_\ell}{n} \sum_{i\in \sC^c } y'_i \partial\eta(\<\bt{x}_\ell, \bt{z}_i\>/n) \bt{z}_j &\text{ when } \ell\in \sC \\
            \frac{y'_\ell}{n} \sum_{i\in \sC} w_i\partial\eta(\<\bt{z}_\ell, \bt{x}_i\>/n)\bt{x}_i &\text{ when } \ell\in \sC^c\, .
        \end{cases}
    \end{align}
    Therefore we have 
    \begin{align}
        \sum_{\ell\in\sC} \lVert \grad_{\bt{z}_\ell}  \<\b{w}, \btC\b{y}'\> \rVert^2 &\le %
        \sum_{\ell\in\sC} \frac{w_\ell^2}{n^2} \norm{\bZ}^2  \sum_{i\in \sC^c} \big(y'_i \partial\eta(\bt{x}_\ell, \bt{z}_\ell)  \big)^2 \\
        &\le \frac{A\norm{\bZ}^2}{pn^2} \max_{\ell \in \sC}\sum_{i\in \sC^c} \partial\eta(\<\bt{x}_\ell, \bt{z}_i\>/n ), \label{eq:gradzlsCbound}
    \end{align}
    where we used the fact that $\abs{y'_i}\le \sqrt{A/p}$ and that $\partial\eta(\cdot)\in \{0, 1\}$. 
    Similarly, for $\ell\in \sC^c$:
    \begin{align}
        \sum_{\ell\in \sC^c} \lVert \grad_{\bt{z}_\ell}\<\b{w}, \btC\b{y}'\>\rVert^2 %
        &\le \sum_{\ell \in \sC^c} \frac{(y'_i)^2 \norm{\bX_\sC}^2}{n^2} \sum_{i\in \sC} \big(w_i\partial\eta(\<\bt{z}_\ell, \bt{x}_i\>/n)\big)^2 \\
        &= \sum_{i\in \sC} \frac{w_i^2 \norm{\bX_\sC}^2}{n^2} \sum_{\ell\in\sC^c} (y'_\ell)^2 \partial\eta(\<\bt{z}_\ell, \bt{x}_i\>/n)^2 \\
        &\le \frac{A\norm{\bX_\sC}^2}{pn^2} \max_{\ell\in \sC} \sum _{i\in \sC^c}\partial\eta (  \<\bt{z}_i, \bt{x}_\ell\>/n). \label{eq:gradzlsCcbound}
    \end{align}
Combining the bounds in Eqs.\eqref{eq:gradzlsCbound}, \eqref{eq:gradzlsCcbound},  we obtain
\begin{align}
        \lVert \grad_{\bZ} \<\b{w}, \btC\b{y}'\>\rVert_F^2
        & = \sum_{\ell\in [p]} \norm{ \grad_{\bt{z}_\ell}\<\b{w}, \btC\b{y}'\>}^2 \\ 
        &\le
        \frac{2A}{pn^2} (\norm{\bX_\sC}^2 + \norm{\bZ}^2)\max_{i\in\sC} \sum_{j\in \sC^c}%
        \partial\eta(\<\bt{x}_i, \bt{z}_j\>/n ). \label{eq:crossgradbound}
    \end{align}
 With $K=C\beta\vee 1$, we define the good set $\cG_5$ of pairs $(\bZ, \bZ')$
    satisfying
    \begin{align}
        \norm{\bZ}, \norm{\bZ'}&\le \sqrt{2n} + \sqrt{2p}\label{eq:cG5defZnorm} \\
        \forall i\in \sC, \quad \frac{1}{p}\sum_{j\in \sC^c} \ind( \<\bt{x}_i, \bt{z}_j\> \ge \tau\sqrt{n}/2)
        &\le 2\exp(-\tau^2/K) \label{eq:cG5defl0norm}\\
        \forall i\in \sC, \quad \frac{1}{p}\sum_{j\in \sC^c} \ind( \<\bt{x}'_i, \bt{z}'_j\> \ge \tau\sqrt{n}/2)
        &\le 2\exp(-\tau^2/K ) \label{eq:cG5defprimel0norm}\\
        \forall i\in \sC, \quad \frac{1}{p}\sum_{j\in \sC^c} \ind( \<\bt{x}'_i, \bt{z}_j\> \ge \tau\sqrt{n}/4)
        &\le 2\exp(-\tau^2/K) \label{eq:cG5defcrossl0norm}\\
        \forall i\in \sC, \quad \frac{1}{p}\sum_{j\in \sC^c} \ind( \<\bt{x}_i, \bt{z}'_j\> \ge \tau\sqrt{n}/4)
        &\le 2\exp(-\tau^2/K ). \label{eq:cG5defcrossl0norm2}
    \end{align}
Define $\bZ(t) = \sqrt{t}\bZ+\sqrt{1-t}\bZ'$ with $(\bZ, \bZ')\in\cG_5$.
 By \myeqref{eq:crossgradbound} the gradient evaluated at     $ \bZ(t) $ is bounded by
    \begin{align}
        \lVert \grad \<\b{w}, \btC\b{y}\>\rVert  ^2
        &\le \frac{2A}{pn^2} (\norm{\bX_\sC (t)}^2 +\norm{\bZ(t)}^2)\max_{i\in\sC} \sum_{j\in \sC^c}%
        \partial\eta(\<\bt{x}(t)_i), \bt{z}(t)_j\>/n )\\
& \le \frac{CA}{p n^2}((\beta\vee 1) n+p) \max_{i\in\sC} \sum_{j\in \sC^c}%
        \partial\eta(\<\bt{x}(t)_i), \bt{z}(t)_j\>/n )\, ,
    \end{align}
where we bounded $\norm{\bX_\sC (t)}$ as in  Eq.~(\ref{eq:BoundBx}),   
and used $\|\bZ(t)\|_{op}\le 2(\sqrt{n}+\sqrt{p})$, which follows from Eq.~(\ref{eq:cG5defZnorm}) and triangle inequality.
 Furthermore, as
    $\<\bt{x}(t)_i, \bt{z}(t)_j\> = t\<\bt{x}_i, \bt{z}_j\> + (1-t)\<\bt{x}'_i, \bt{z}'_j\> +
    \sqrt{t(1-t)} (\<\bt{x}_i, \bt{z}'_j\> + \<\bt{x}'_i, \bt{z}_j\>)$, we have that:
    \begin{align}
        \partial\eta(\<\bt{x}(t)_i), \bt{z}(t)_j\>/n ) & = \ind( \abs{\<\bt{x}(t)_i, \bt{z}(t)_j\>} \ge \tau\sqrt{n}) \\
        &\le  \ind (  | \<\bt{x}_i, \bt{z}_j\>| \ge \tau\sqrt{n}/2) + \ind(|\<\bt{x}'_i, \bt{z}'_j\>| \ge \tau\sqrt{n}/2) \nonumber \\
        &\quad+ \ind(|\<\bt{x}'_i, \bt{z}_j\>| \ge \tau\sqrt{n}/4)  + \ind(|\<\bt{x}'_i, \bt{z}_j\>|   \ge \tau\sqrt{n}/4).
    \end{align}
    Hence on the good set $\cG_5$, we have:
    \begin{align}
        \max_{i\in \sC}\sum_{j\in \sC^c}\partial\eta(\<\bt{x}(t)_i), \bt{z}(t)_j\>/n ) &\le 4p \, e^{- \tau^2/ K}\, .
    \end{align}
    Therefore the gradient satisfies, on the good set:
    \begin{align}
        \lVert \grad_{\bZ} \<\b{w}, \btC\b{y}\>\rVert  ^2
        &\le C\frac{A}{n^2} ((\beta\vee 1)n+p) \, e^{-\tau^2/K}  = CL_*^2\, .
    \end{align}
    Hence, by Lemma \ref{lem:nonconvexConc}, we obtain:
    \begin{align}
        \tprob\Big\{ \max_{\b{w}\in T^{\eps}_{\sC}, \by\in T^\eps_{p}} \<\b{w}, \btC\by'\> \ge \Delta \Big\}
        \le & C|T^\eps_{\sC}||T^\eps_{p}|\exp\Big(-\frac{\Delta^2}{C L_*^2} \Big) \label{eq:concLemmaApp}\\
& + CL_*^{-2}\tE\{\max \<\b{w}, \btC\b{y}'\>^4 \}^{1/4}\P\{\cG_5^c\}^{1/2}\, . \nonumber
    \end{align}
    By Lemma \ref{lem:epsnetcard}, keeping $\eps=1/4$ we have that the first term is at most $C \exp(Cp + \exp(-\Delta^2/CL_*^2))$.
    For the second term, we have $|\<\b{w}, \btC\b{y}\>| \le \norm{\btC}\le \norm{\btC}_F \le \norm{\hbSigma}_F$. Since
    $\E\{\lVert \hbSigma\rVert_F^4\} \le (np)^C$, 
    we have that $\E\{\max_{\b{w}, \b{y}}\<\b{w}, \btC\b{y}\>^4\}^{1/4} \le (np)^C$.
    Also as $\tau <\sqrt{\log p}$, $L_* \ge (np)^{-C}$, implying that the
    second term is bounded above by $(np)^{C} \P\{\cG_5^c\}^{1/2}$. Therefore:
    \begin{align}
        \tprob\Big\{ \max_{\b{w}\in T^{\eps}_{\sC}, \by\in T^\eps_{p}} \<\b{w}, \btC\by'\> \ge \Delta \Big\}%
        &\le C\exp\Big(Cp - \frac{\Delta^2}{C L_*^2} \Big) + (np)^C\P\{\cG_5^c\}^{1/2}\, . %
        \label{eq:concLemmaApp2}
    \end{align}

    It remains to control the probability of the
    bad set $\cG_5^c$. For this, we control the probability of violating
    any one condition among \eqref{eq:cG5defZnorm}, \eqref{eq:cG5defl0norm},
    \eqref{eq:cG5defprimel0norm}, \eqref{eq:cG5defcrossl0norm} and \eqref{eq:cG5defcrossl0norm2} defining
    $\cG_5$ and then use the union bound. 
    By Lemmas \ref{lem:gaussianmatnorm}, condition     \eqref{eq:cG5defZnorm} hold with probability 
    $1- C\exp(-cn)$. 
    The argument controlling the probability for conditions \eqref{eq:cG5defl0norm},
    \eqref{eq:cG5defprimel0norm}, \eqref{eq:cG5defcrossl0norm} and \eqref{eq:cG5defcrossl0norm2} to hold
    are essentially the same, so we restrict ourselves to condition \eqref{eq:cG5defl0norm} keeping $i = 1 \in \sC$,
    without loss of generality. 
    Conditional on $\bt{x}_1$, $\<\bt{x}_1, \bt{z}_j\>$ for $j\in \sC^c$
    are independent $\normal(0, \norm{\bt{x}_1}^2)$ variables. Therefore, conditional on $\bt{x}_1$, 
    $\ind( | \<\bt{x}_1, \bt{z}_j\>|\ge \tau\sqrt{n}/2)$ are independent Bernoulli random variables
    with success probability $\Phi\{-\tau\sqrt{n}/2\norm{\bt{x}_1} \}$. Define $h_1$ to be the success probability, i.e.
    $h_1 = \Phi(-\tau\sqrt{n}/(2\norm{\bt{x}_1}))$.    

Since $K = C(\beta\vee 1)$ we can enlarge $C$ to a large absolute constant. 
Letting $\b{V}\in\reals^{n\times r}$ be the matrix with 
columns $\bv_1,\dots,\bv_r$, and $\b{B}$ the diagonal matrix with $B_{q,q} = \sqrt{\beta_q}$, we have,
with probability at least $1-\exp(-n/C)$,
\begin{align}
\|\bt{x}_1\|\le \|\bU\b{B}\b{V}^{\sT}\b{e_1}\|+\|\bt{z}_1\|\le \|\b{B}\|\|\bU\|+\|\bt{z}_1\|\le \sqrt{\frac{Kn}{4}}\, ,
\end{align}
where the last equality holds since $\bU\in\cU$ and by tail bounds on chi-squared random variables.
Further
    \begin{align}
        \tprob \Big\{  \sum_{j\in \sC^c} \ind( |\<\bt{x}_1, \bt{z}_j\> \ge \tau\sqrt{n}/2) \ge |\sC^c| h \Big\} &\le \tprob\{\norm{\bt{x}_1}^2 \ge Kn\}  \nonumber\\%
        &+\sup_{\norm{\bt{x}_1}^2 \le Kn } \tprob \Big\{  \sum_{j\in \sC^c} \ind( |\<\bt{x}_1, \bt{z}_j\> \ge \tau\sqrt{n}/2) \ge |\sC^c| h  \;\Big|\; \bt{x}_1\Big\}. 
\label{eq:conditionalBound} 
    \end{align}
By the above argument, the first term is at most $\exp(-n/C)$ and
    we turn to the second term. By the Chernoff bound
    \begin{align}
  \tprob \Big\{  \sum_{j\in \sC^c} \ind( |\<\bt{x}_1, \bt{z}_j\> \ge \tau\sqrt{n}/2) \ge \abs{\sC^c}h  \,\,\,\big| \bt{x}_1 \Big\} &\le \exp\big(- |\sC^c| D(h||h_1)\big), ,
    \end{align}
with $h_1<\exp(-\tau^2/K)$ when $\norm{\bt{x}_1}^2 \le Kn/4$. Choosing  $ h = 2\exp(-\tau^2/K)$ implies that $h_1  \le  h/2$ when
   and, thereby, 
    that $D(h\Vert h_1) \ge h/C$. Further since $\tau <\sqrt{\log p}/2$, $h \ge 1/\sqrt{p}$. This implies that 
    \begin{align}
        \exp(-|\sC^c| D(h-h_1 \Vert h_1)) = \exp(-(p-s_0)h/C) \ge \exp(-\sqrt{p}/C).
    \end{align}
    Combining this with \myeqref{eq:conditionalBound} we have that $\P\{\cG^c\} \le Cp^2\exp(-\min(n, \sqrt{p})/C)$ for some absolute $C$. 
    Plugging this in \myeqref{eq:concLemmaApp2} yields the lemma.
\end{proof}

We are now ready to prove Proposition \ref{prop:cross}. Indeed, as in Proposition \ref{prop:noise}, for any unit vector $\b{y}\in\reals^{\sC^c}$, 
let $\sS = \{i: \abs{y_i} \ge \sqrt{A/p}\}$ and $\b{y}_{\sS}, \b{y}_{\sS^c}$ denote the projections 
on the indices in $\sS, \sS^c$ respectively. 
\begin{align}
    \tprob\Big\{ \norm{\btC_1} \ge \Delta\} &\le \tprob\Big\{ \max_{\b{w}\in T^\eps_\sC, \b{y}\in T^\eps_{\sC^c}} |\<\b{w}, \btC\b{y}\>| \ge \Delta(1-2\eps) \Big\} \\
    &\le \tprob\Big\{ \max_{\b{w}\in T^\eps_{\sC}, \b{y}\in T^\eps_{\sC^c} } |\<\b{w}, \btC\b{y}_{\sS}\>| \ge \Delta(1-2\eps)/2\Big\} \nonumber \\&\quad+
    \P\Big\{ \max_{\b{w}\in T^\eps_{\sC}, \b{y}\in T^\eps_{\sC^c} } |\<\b{w}, \btC\b{y}_{\sS^c}\>| \ge \Delta(1-2\eps)/2\Big\}.
\end{align}
As before, we will let $\eps=1/4$. 
The first term is controlled via Lemma \ref{lem:smallSetCross}, while the second is controlled
by Lemma \ref{lem:spreadOutCross}. We keep $\Delta = \Delta_*$ where
\begin{align}
    \Delta_* &= C\Big( L_* \sqrt{p} + L \sqrt{  \frac{ p\log A}{ A}}\Big).
\end{align}
so that, via the bounds of Lemmas \ref{lem:smallSetCross}, \ref{lem:spreadOutCross} and that $s_0 ^2 \le p$:
\begin{align}
    \P\{\norm{\bC_1} \ge \Delta_* \} &\le C \exp\Big( -c \frac{p\log A}{A} \Big) + L_*^{-2} (np)^C\exp\big(-c \min(\sqrt{p}, n)\big).   
\end{align}
We now set $A = \big((\tau^2/K) \exp(\tau^2/K)\big)^{1/2}$  with $K = C(\beta\vee 1)$ for a suitable constant $C$
and, since $\tau\le \sqrt{\log p}/2$, we get that $A \le p^{1/3}$. Furthermore, it is straightforward to see that $L\ge (np)^{-C}$, 
and this implies that
\begin{align}
    \P\{\norm{\bC_1} \ge \Delta_* \} &\le (np)^{C}\exp(-c\min(\sqrt{p}, n)) = o(1).
\end{align}
With this setting of $A$, we get the form of $\Delta_*$ below, as required for the proposition. 
\begin{align}
    \Delta_* &\le C\,e^{-c\tau^2/K} \sqrt{\frac{\tau^2 \vee 1}{K} \cdot \frac{pn(\beta\vee 1)+ p^2}{n^2} } \\
    & \le C\,(\tau\vee 1) e^{-c\tau^2/K} \sqrt{\frac{p}{n}} \vee \frac{p}{n}.  
\end{align}

\subsection{Proof of Proposition \ref{prop:diag}}\label{subsec:proofdiag}

Since $\b{D}$ is a diagonal matrix, its spectral norm is bounded by the maximum of its entries. 
This is easily done as, for every $i\in \sC^c$:
\begin{align}
    \abs{(\b{D})_{ii}} &= \left|\eta\Big( \frac{\norm{\btz_i}^2}{n} - 1;\frac{\tau}{\sqrt{n}} \Big)\right| \\
  &\le \Big|\frac{\norm{\btz_i}^2-n}{n}\Big|\,.
\end{align}
By the Chernoff bound for $\chi^2$-squared random variables as in Lemma \ref{lem:chisquare} followed by the union bound,
with probability $1- o(1)$:
\begin{align}
    \max_{i}\Big|\frac{\norm{\btz_i}^2}{n}-1\Big|\le C\sqrt{\frac{\log p}{n}}
\end{align}
for some absolute $C$. Here we used the fact that $(\log p) /n < 1$. 

\subsection{Proof of Proposition \ref{prop:largeThresh}}\label{subsec:prooflargethresh}

It suffices to show that with probability $1-o(1)$
\begin{align}
    \max_{i, j \in \sF\cup \sG} |\widehat{\Sigma}_{ij}|&\le \frac{\tau}{\sqrt{n}} = C_0(\beta\vee 1)\sqrt{\frac{\log p}{n}}. 
\end{align}
This is a standard argument \cite[Lemma A.3]{bickel2008regularized} where
(following the dependence on $\beta$) it suffices to take $\tau \ge C_0 (\beta\vee 1)\sqrt{\log p}$ for $C_0$ a sufficiently large absolute constant. We note here that the same can also
be proved via the conditioning technique applied in the
proofs of Propositions \ref{prop:noise} and \ref{prop:cross}.

\section{Proof of Theorems \ref{thm:main}}\label{sec:proofmain}

Throughout this section, to lighten notation, we drop the prime from
$\hbSigma'$ and $\bX'$ while keeping in mind that these are independent from
$\hbv_1,\dots,\hbv_r$. We further write $\bX = \bU\bB\bV^{\sT}+\bZ$, where $\bU\in\reals^{n\times r}$
is the matrix with columns $\bu_1,\dots,\bu_r$, $\bB$ is diagonal with $B_{ii}=\sqrt{\beta_i}$ and
$\bV\in\reals^{p\times r}$ has columns $\bv_1,\dots,\bv_r$. 

Define the event
\begin{align}
\cU \equiv \Big\{\bU\in\reals^{n\times r} :\;\; \Big\|\frac{1}{n}\bU^{\sT}\bU -\id_{r\times r}\Big\|_{op}\le 3\sqrt{\frac{r}{n}} \Big\}\, .
\end{align}
By the Bai-Yin law on eigenvalues of Wishart matrices \cite{Vershynin-CS}, 
$\lim_{n\to\infty}\prob(\bU\in \cU)=1$. In the rest of the proof, we will therefore assume $\bU\in\cU$ fixed,
and denote by $\tprob(\,\cdot\,) = \prob(\,\cdot\, |\bU)$ the expectation conditional on $\bU$. In other words,
$\tprob(\,\cdot\,)$ denotes expectation with respect to $\bZ$. 

Note that
\begin{align}
\hbSigma = \frac{1}{n} \bV\bB\bU^{\sT}\bU\bB\bV^{\sT}+\frac{1}{n}
\bZ^{\sT}\bU\bB\bV^{\sT}+\frac{1}{n}\bV\bB\bU^{\sT}\bZ+ \frac{1}{n}\bZ^{\sT}\bZ-\id\, .
\end{align}
We then have, for $q\in \{1,\dots, r\}$ and $i\in\{1,\dots,p\}$,
\begin{align}
\big|(\hbSigma\hbv_q)_i&-\beta_q\<\bv_q,\hbv_q\>\,v_{q,i}\big|\le T^{(1)}_{i,q} +T^{(2)}_{i,q}+T^{(3)}_{i,q}\, ,\\
T^{(1)}_{i,q}&\equiv \Big|\frac{1}{n}\<\bfe_i,\bV\bB\bU^{\sT}\bU\bB\bV^{\sT}\hbv_q\>-\beta_q\<\bv_q,\hbv_q\>v_{q,i}\Big|\, ,\\
T^{(2)}_{i,q}& \equiv \frac{1}{n}\,\Big|\<\bZ,\big[(\bU\bB\bV^{\sT}\bfe_i) \hbv_q^{\sT}+(\bU\bB\bV^{\sT}\hbv_q)\bfe_i^{\sT}\big]\>\Big|\, ,
\\
T^{(3)}_{i,q}& \equiv \Big|\<\bfe_i,\Big(\frac{1}{n}\bZ^{\sT}\bZ-\id\Big)\hbv_q\>\Big|\, .
\end{align}
We next bound, with high probability, $\max_{i,q} T^{(a)}_{i,q}$ for $a\in\{1,2,3\}$.
Throughout we let $\eps \equiv \max_{q\in [r]}\|\hbv_q-\bv_q\|$.

Considering the first term, we have
\begin{align}
T^{(1)}_{i,q}&\le \Big|\<\bfe_i,\bV\bB\Big(\frac{1}{n}\bU^{\sT}\bU-\id\Big)\bB\bV^{\sT}\hbv_q\>\Big|+
\Big|\<\bfe_i,\bV\bB^2\bV^{\sT}\hbv_q\>-\beta_q\<\bv_q,\hbv_q\>v_{q,i}\Big|\\
& \le 2\beta\sqrt{\frac{r}{n}}+\beta\eps\sqrt{r}\max_{q'\in[r]\setminus q}\, |v_{q',i}|\, ,\label{eq:Bound1}
\end{align}
where in the last inequality we used $\sum_{q'\in [r]\setminus q}\<\bv_{q'},\hbv_q\>^2 \le 1-\<\bv_{q},\hbv_q\>^2\le \eps^2/2$.

Consider next the second term. Since $Z_{ij}\sim_{iid}\normal(0,1)$, it follows that 
$T^{(2)}_{i,q}=|W_{i,q}|$, for $W_{i,q}\sim\normal(0,\sigma^2_{i,q})$  a Gaussian random variable with variance
\begin{align}
\sigma^2_{i,q} &= \frac{1}{n^2}\big\|(\bU\bB\bV^{\sT}\bfe_i) \hbv_q^{\sT}+(\bU\bB\bV^{\sT}\hbv_q)\bfe_i^{\sT}\big\|_F^2\\
& \le \frac{2}{n^2}\Big\{\|\bU\bB\bV^{\sT}\bfe_i\|^2 +\|\bU\bB\bV^{\sT}\hbv_q\|^2\Big\}\\
& \le \frac{4}{n^2}\|\bU\bB\bV^{\sT} \|_{op}^2 \\
& \le \frac{4}{n^2}\|\bU\|_{op}^2\|\bB\|_{op}^2 \le \frac{8\beta^2}{n}\, .
\end{align}
By union bound over $i\in [p]$, $q\in [r]$ we obtain
\begin{align}
\max_{i\in [p],q\in [r]}T^{(2)}_{i,q}\le 8\beta\sqrt{\frac{\log p}{n}}\, . \label{eq:Bound2}
\end{align}

Finally, consider the last term. By rotational invariance of $\bZ$, the distribution of $T^{(3)}_{i,q}$
only depends on the angle between $\bfe_i$ and $\hbv_q$. Calling this angle $\vartheta$,
we have
\begin{align}
T^{(3)}_{i,q}&\ed\Big|\<\bfe_1, \Big(\frac{1}{n}\bZ^{\sT}\bZ-\id\Big)\bfe_1\>\cos\vartheta +
\<\bfe_1, \Big(\frac{1}{n}\bZ^{\sT}\bZ-\id\Big)\bfe_2\>\sin\vartheta\Big|\\
&\le \Big|\frac{1}{n}\|\btz_1\|^2-1\Big|+\Big|\frac{1}{n}\<\btz_1,\btz_2\>\Big|\, .
\end{align}
Both of these terms have Bernstein-type tail bonds, whence
\begin{align}
\tprob\Big(T^{(3)}_{i,q}\ge \frac{t}{\sqrt{n}}\Big)\le 2\, \exp\big\{-c\min(t\sqrt{n},t^2)\big\}\, .
\end{align}
Using $t = C_0\sqrt{\log p}$, and recalling that $n\ge C\log p$ for $C$ a large constant, we
obtain  $\tprob\big(T^{(3)}_{i,q}\ge C_0\sqrt{(\log p)/n}\big)\le 2\, p^{-10}$. Hence by union bound
\begin{align}
\max_{i\in [p],q\in [r]}T^{(3)}_{i,q}\le C_0\sqrt{\frac{\log p}{n}}\, . \label{eq:Bound3}
\end{align}

By putting together Eqs.~(\ref{eq:Bound1}), (\ref{eq:Bound2}), (\ref{eq:Bound3}), 
and using assumption {\sf A2}, we get
\begin{align}
\big|(\hbSigma\hbv_q)_i-\beta_q\<\bv_q,\hbv_q\>\,v_{q,i}\big|&\le C\beta\sqrt{\frac{r}{n}}+
C(\beta\vee 1)\sqrt{\frac{\log p}{n}} + \beta\eps\gamma\sqrt{r} \, |v_{q,i}|\; \ind(i\in\sC)\, .
\end{align}
Let $\hsC_q = \{i\in [p]: \; |(\hbSigma'\hbv_q)_i|\ge \rho \}$.
We claim that the above implies that, with high probability, $\sC_q \subseteq \hsC_q\subseteq \sC$ for all $q$.

For  $i\not\in \sC$, we have
\begin{align}
\big|(\hbSigma\hbv_q)_i\big|&\le C\beta\sqrt{\frac{r}{n}}+
C(\beta\vee 1)\sqrt{\frac{\log p}{n}} \\
& < \frac{\betamin\theta}{2\sqrt{s_0}}\, ,
\end{align}
where the last inequality follows from Eq.~(\ref{eq:Nsupport}). 

On the other hand,  By Theorem \ref{thm:corr} and using the assumption (\ref{eq:Nsupport}),
we can guarantee 
\begin{align}
\eps\le \frac{1}{8}\Big(\frac{\betamin}{\beta\gamma\sqrt{r}}\, \wedge \, 1\,\Big)\, .
\end{align}
Hence for $i\in \sC_q$, and considering --to be definite-- $v_{q,i}>0$, we get
\begin{align}
(\hbSigma\hbv_q)_i &\ge \beta_q\<\bv_q,\hbv_q\>\,v_{q,i} -C\beta\sqrt{\frac{r}{n}}-
C(\beta\vee 1)\sqrt{\frac{\log p}{n}}-\beta\eps\gamma\sqrt{r} \, |v_{q,i}|\\
&\ge \betamin\Big(1-\eps -\frac{\beta}{\betamin}\eps\gamma\sqrt{r}\Big)v_{q,i}
-C\beta\sqrt{\frac{r}{n}}- C(\beta\vee 1)\sqrt{\frac{\log p}{n}}\\
&\ge \frac{3\betamin\theta}{4\sqrt{s_0}} - C\beta\sqrt{\frac{r}{n}}- C(\beta\vee 1)\sqrt{\frac{\log p}{n}}\\
& >\frac{\betamin\theta}{2\sqrt{s_0}}\, .
\end{align}
where, in the first inequality, we used $\<\bv_q,\hbv_q\>\ge 1-\eps$.

This concludes the proof. Keeping track of the dependence on $\theta$, $\gamma$, $\beta$, $\betamin$,
we get that the following conditions are sufficient for the theorem's conclusion to hold (with $C$ a suitable 
numerical constant):
\begin{align}
n & \ge C\, \frac{(\beta^2\vee 1)}{\betamin^2\theta^2}\, s_0 \log p\, ,\\
n & \ge C\,  \frac{\beta^2}{\betamin^2\theta^2}\, r s_0 \, ,\\ 
n & \ge C\, \left\{\frac{\beta^4\vee\beta^2}{\betamin^2}\gamma^2\right\} \, r\, s_0^2\log\frac{p}{s_0^2}\, ,\\
n & \ge C\, \frac{(\beta^2\vee 1)}{\betamin^2}\, s_0^2\log\frac{p}{s_0^2}\, .
\end{align}
All of these conditions are implied by the assumptions of Theorem \ref{thm:main}, namely
Eq.~(\ref{eq:Nsupport}). In particular, this is shown by using the fact that $s_0\log p\le s_0^2\log(p/s_0^2)$
for  $s_0\le \sqrt{p}$.

\section*{Acknowledgements}

We are grateful to David Donoho for his feedback on this manuscript.
This work was
partially supported by the NSF CAREER award CCF-0743978, the NSF grant CCF-1319979, and
the grants AFOSR/DARPA FA9550-12-1-0411 and FA9550-13-1-0036.

\bibliographystyle{amsalpha}
 \bibliography{all-bibliography}

\newcommand{\etalchar}[1]{$^{#1}$}
\providecommand{\bysame}{\leavevmode\hbox to3em{\hrulefill}\thinspace}
\providecommand{\MR}{\relax\ifhmode\unskip\space\fi MR }
\providecommand{\MRhref}[2]{%
  \href{http://www.ams.org/mathscinet-getitem?mr=#1}{#2}
}
\providecommand{\href}[2]{#2}
\begin{thebibliography}{CMW{\etalchar{+}}13}

\bibitem[AW09]{amini2009high}
Arash~A Amini and Martin~J Wainwright, \emph{High-dimensional analysis of
  semidefinite relaxations for sparse principal components}, The Annals of
  Statistics \textbf{37} (2009), no.~5B, 2877--2921.

\bibitem[BBAP05]{baik2005phase}
Jinho Baik, G{\'e}rard Ben~Arous, and Sandrine P{\'e}ch{\'e}, \emph{Phase
  transition of the largest eigenvalue for nonnull complex sample covariance
  matrices}, Annals of Probability (2005), 1643--1697.

\bibitem[BGN11]{benaych2011eigenvalues}
Florent Benaych-Georges and Raj~Rao Nadakuditi, \emph{The eigenvalues and
  eigenvectors of finite, low rank perturbations of large random matrices},
  Advances in Mathematics \textbf{227} (2011), no.~1, 494--521.

\bibitem[BL08a]{bickel2008covariance}
Peter~J Bickel and Elizaveta Levina, \emph{Covariance regularization by
  thresholding}, The Annals of Statistics (2008), 2577--2604.

\bibitem[BL08b]{bickel2008regularized}
\bysame, \emph{Regularized estimation of large covariance matrices}, The Annals
  of Statistics (2008), 199--227.

\bibitem[BR13]{berthet2013computational}
Quentin Berthet and Philippe Rigollet, \emph{Computational lower bounds for
  sparse pca}, arXiv preprint arXiv:1304.0828 (2013).

\bibitem[CDMF09]{capitaine2009largest}
Mireille Capitaine, Catherine Donati-Martin, and Delphine F{\'e}ral, \emph{The
  largest eigenvalues of finite rank deformation of large wigner matrices:
  convergence and nonuniversality of the fluctuations}, The Annals of
  Probability \textbf{37} (2009), no.~1, 1--47.

\bibitem[CL11]{cai2011adaptive}
Tony Cai and Weidong Liu, \emph{Adaptive thresholding for sparse covariance
  matrix estimation}, Journal of the American Statistical Association
  \textbf{106} (2011), no.~494, 672--684.

\bibitem[CMW{\etalchar{+}}13]{cai2013sparse}
T~Tony Cai, Zongming Ma, Yihong Wu, et~al., \emph{Sparse pca: Optimal rates and
  adaptive estimation}, The Annals of Statistics \textbf{41} (2013), no.~6,
  3074--3110.

\bibitem[CS13]{cheng2013spectrum}
Xiuyuan Cheng and Amit Singer, \emph{The spectrum of random inner-product
  kernel matrices}, Random Matrices: Theory and Applications \textbf{2} (2013),
  no.~04, 1350010.

\bibitem[CZ{\etalchar{+}}12]{cai2012optimal}
T~Tony Cai, Harrison~H Zhou, et~al., \emph{Optimal rates of convergence for
  sparse covariance matrix estimation}, The Annals of Statistics \textbf{40}
  (2012), no.~5, 2389--2420.

\bibitem[CZZ{\etalchar{+}}10]{cai2010optimal}
T~Tony Cai, Cun-Hui Zhang, Harrison~H Zhou, et~al., \emph{Optimal rates of
  convergence for covariance matrix estimation}, The Annals of Statistics
  \textbf{38} (2010), no.~4, 2118--2144.

\bibitem[dBG08]{d2008optimal}
Alexandre d'Aspremont, Francis Bach, and Laurent~El Ghaoui, \emph{Optimal
  solutions for sparse principal component analysis}, The Journal of Machine
  Learning Research \textbf{9} (2008), 1269--1294.

\bibitem[dEGJL07]{d2007direct}
Alexandre d'Aspremont, Laurent El~Ghaoui, Michael~I Jordan, and Gert~RG
  Lanckriet, \emph{A direct formulation for sparse pca using semidefinite
  programming}, SIAM review \textbf{49} (2007), no.~3, 434--448.

\bibitem[DJ94]{DJ94a}
David~L. Donoho and Iain~M. Johnstone, \emph{Minimax risk over $l_p$ balls},
  Prob. Th. and Rel. Fields \textbf{99} (1994), 277--303.

\bibitem[DK70]{davis1970rotation}
Chandler Davis and William~Morton Kahan, \emph{The rotation of eigenvectors by
  a perturbation. iii}, SIAM Journal on Numerical Analysis \textbf{7} (1970),
  no.~1, 1--46.

\bibitem[DM14]{deshpande2014sparse}
Yash Deshpande and Andrea Montanari, \emph{Sparse pca via covariance
  thresholding}, Advances in Neural Information Processing Systems, 2014,
  pp.~334--342.

\bibitem[EK10a]{el2010information}
Noureddine El~Karoui, \emph{On information plus noise kernel random matrices},
  The Annals of Statistics \textbf{38} (2010), no.~5, 3191--3216.

\bibitem[EK10b]{el2010spectrum}
\bysame, \emph{The spectrum of kernel random matrices}, The Annals of
  Statistics \textbf{38} (2010), no.~1, 1--50.

\bibitem[FM15]{fan2015spectral}
Zhou Fan and Andrea Montanari, \emph{The spectral norm of random inner-product
  kernel matrices}, {\sf arXiv:1507.05343} (2015).

\bibitem[FP07]{feral2007largest}
Delphine F{\'e}ral and Sandrine P{\'e}ch{\'e}, \emph{The largest eigenvalue of
  rank one deformation of large wigner matrices}, Communications in
  mathematical physics \textbf{272} (2007), no.~1, 185--228.

\bibitem[JL04]{johnstone2004sparse}
Iain~M Johnstone and Arthur~Yu Lu, \emph{Sparse principal components analysis},
  Unpublished manuscript (2004).

\bibitem[JL09]{johnstone2009consistency}
\bysame, \emph{On consistency and sparsity for principal components analysis in
  high dimensions}, Journal of the American Statistical Association
  \textbf{104} (2009), no.~486.

\bibitem[Joh15]{johnstone2013function}
Iain~M. Johnstone, \emph{Function estimation and gaussian sequence models},
  Unpublished manuscript (2015).

\bibitem[Kar08]{karoui2008operator}
Noureddine~El Karoui, \emph{Operator norm consistent estimation of
  large-dimensional sparse covariance matrices}, The Annals of Statistics
  (2008), 2717--2756.

\bibitem[KNV13]{KrauthgamerSPCA}
R.~Krauthgamer, B.~Nadler, and D.~Vilenchik, \emph{Do semidefinite relaxations
  really solve sparse pca?}, CoRR \textbf{abs/1306:3690} (2013).

\bibitem[KY13]{knowles2013isotropic}
Antti Knowles and Jun Yin, \emph{The isotropic semicircle law and deformation
  of wigner matrices}, Communications on Pure and Applied Mathematics (2013).

\bibitem[Led01]{Ledoux}
M.~Ledoux, \emph{{The concentration of measure phenomenon}}, Mathematical
  Surveys and Monographs, vol.~89, {American Mathematical Society, Providence,
  RI}, 2001.

\bibitem[LM00]{laurent2000adaptive}
Beatrice Laurent and Pascal Massart, \emph{Adaptive estimation of a quadratic
  functional by model selection}, Annals of Statistics (2000), 1302--1338.

\bibitem[MB06]{meinshausen2006high}
Nicolai Meinshausen and Peter B{\"u}hlmann, \emph{High-dimensional graphs and
  variable selection with the lasso}, The Annals of Statistics (2006),
  1436--1462.

\bibitem[MW15a]{ma2015sum}
Tengyu Ma and Avi Wigderson, \emph{Sum-of-squares lower bounds for sparse pca},
  Advances in Neural Information Processing Systems, 2015, pp.~1603--1611.

\bibitem[MW{\etalchar{+}}15b]{ma2015computational}
Zongming Ma, Yihong Wu, et~al., \emph{Computational barriers in minimax
  submatrix detection}, The Annals of Statistics \textbf{43} (2015), no.~3,
  1089--1116.

\bibitem[MWA05]{moghaddam2005spectral}
Baback Moghaddam, Yair Weiss, and Shai Avidan, \emph{Spectral bounds for sparse
  pca: Exact and greedy algorithms}, Advances in neural information processing
  systems, 2005, pp.~915--922.

\bibitem[Pau07]{paul2007asymptotics}
Debashis Paul, \emph{Asymptotics of sample eigenstructure for a large
  dimensional spiked covariance model}, Statistica Sinica \textbf{17} (2007),
  no.~4, 1617.

\bibitem[Ver12]{Vershynin-CS}
R.~Vershynin, \emph{Introduction to the non-asymptotic analysis of random
  matrices}, Compressed Sensing: Theory and Applications (Y.C. Eldar and
  G.~Kutyniok, eds.), Cambridge University Press, 2012, pp.~210--268.

\bibitem[VL12]{vu2012minimax}
Vincent~Q Vu and Jing Lei, \emph{Minimax rates of estimation for sparse pca in
  high dimensions}, Proceedings of the 15th International Conference on
  Artificial Intelligence and Statistics (AISTATS) 2012, 2012.

\bibitem[Wai09]{wainwright2009sharp}
Martin~J Wainwright, \emph{Sharp thresholds for high-dimensional and noisy
  sparsity recovery using-constrained quadratic programming (lasso)},
  Information Theory, IEEE Transactions on \textbf{55} (2009), no.~5,
  2183--2202.

\bibitem[WBS14]{wang2014statistical}
Tengyao Wang, Quentin Berthet, and Richard~J Samworth, \emph{Statistical and
  computational trade-offs in estimation of sparse principal components}, arXiv
  preprint arXiv:1408.5369 (2014).

\bibitem[WTH09]{witten2009penalized}
Daniela~M Witten, Robert Tibshirani, and Trevor Hastie, \emph{A penalized
  matrix decomposition, with applications to sparse principal components and
  canonical correlation analysis}, Biostatistics \textbf{10} (2009), no.~3,
  515--534.

\bibitem[ZHT06]{zou2006sparse}
Hui Zou, Trevor Hastie, and Robert Tibshirani, \emph{Sparse principal component
  analysis}, Journal of computational and graphical statistics \textbf{15}
  (2006), no.~2, 265--286.

\end{thebibliography}
 
\end{document}